\documentclass[11pt,leqno,oneside]{amsart}
\usepackage{amsmath,amssymb,amsxtra,comment,graphicx,psfrag}
\usepackage{bm,mathrsfs}
\usepackage{mathtools}
\usepackage{xspace}
\usepackage{color}
\usepackage{enumerate}
\usepackage{subfigure,epstopdf}
\usepackage{pdfsync}
\usepackage{hhline}
\usepackage[margin=1.0in]{geometry}
\usepackage{fancyhdr}
\allowdisplaybreaks

\newcommand{\al}{\alpha}
\newcommand{\fy}{\varphi}

\def\Dal{{_0\partial_t^\al}}
\def\dDal{{_t\partial_T^\al}}
\def\bPtau{{_0\bar\partial_\tau^\al}}

\def\Om{\Omega}
\def\II{(\Om)}

\def\Uad{U_{\rm ad}}

\def\d{{\rm d}}

\def\vv{\vert\thickspace\!\!\vert\thickspace\!\!\vert}
\theoremstyle{plain}% default
\newtheorem{theorem}{Theorem}[section]
\newtheorem{remark}{Remark}[section]

\newtheorem{lemma}{Lemma}[section]

\numberwithin{equation}{section}
\usepackage{titlesec}
\titleformat{\section}{\vskip10pt\normalsize\bfseries}{\thesection.}{0.5em}{\centering}
\titleformat{\subsection}{\vskip10pt\normalsize\bfseries}{\thesubsection.}{0.5em}{}
\def\d{{\rm d}}

\begin{document}

\title[]{Pointwise-in-time error estimates for an optimal control problem with subdiffusion constraint}

\author[Bangti Jin]{Bangti Jin}
\address{Department of Computer Science, University College London, Gower Street, London, WC1E 2BT, UK.}
\email {{b.jin@ucl.ac.uk}}

\author[Buyang Li]{$\,\,$Buyang Li$\,$}
\address{Department of Applied Mathematics,
The Hong Kong Polytechnic University, Kowloon, Hong Kong}
\email {{buyang.li{\it @\,}polyu.edu.hk}}

\author[Zhi Zhou]{$\,\,$Zhi Zhou$\,$}
\address{Department of Applied Mathematics,
The Hong Kong Polytechnic University, Kowloon, Hong Kong}
\email {{zhizhou{\it @\,}polyu.edu.hk}}

%\date{}
\date{\today}

\maketitle

\begin{abstract}
{In this work, we present numerical analysis for a distributed optimal control problem, with box constraint on
the control, governed by a subdiffusion equation which involves a fractional
derivative of order $\alpha\in(0,1)$ in time. The fully discrete scheme is obtained by applying the conforming
linear Galerkin finite element method in space, L1 scheme/backward Euler convolution quadrature in time, and
the control variable by a variational type discretization. With a space mesh size $h$ and time stepsize $\tau$,
we establish the following order of convergence for the numerical solutions of the optimal control problem:
$O(\tau^{\min({1}/{2}+\alpha-\epsilon,1)}+h^2)$ in the discrete $L^2(0,T;L^2(\Omega))$ norm and
$O(\tau^{\alpha-\epsilon}+\ell_h^2h^2)$ in the discrete $L^\infty(0,T;L^2(\Omega))$ norm, with any
small $\epsilon>0$ and $\ell_h=\ln(2+1/h)$. The analysis relies essentially on the maximal
$L^p$-regularity and its discrete analogue for the subdiffusion problem. Numerical experiments are
provided to support the theoretical results.
}
% Keywords:
{optimal control, time-fractional diffusion, L1 scheme, convolution quadrature, pointwise-in-time
error estimate, maximal regularity.}
\end{abstract}

%%%%%%%%%%%%%%%%%%%%%%%%%%%%%%%%%%%%%%%%%%%%%%%%%%
\section{Introduction}
Let $\Omega\subset\mathbb{R}^d $ ($d=1,2,3$) be a convex polyhedral domain with a boundary $\partial\Omega$.
Consider the distributed optimal control problem
\begin{equation}\label{eqn:ob}
    \min_{q \in \Uad} J(u,q)=\tfrac12 \| u - u_d  \|_{L^2(0,T;L^2\II)}^2 + \tfrac\gamma2\| q \|_{L^2(0,T;L^2\II)}^2,
\end{equation}
subject to the following fractional-order partial differential equation
\begin{align}\label{eqn:fde}
%\left\{\begin{aligned}
\Dal u-\Delta u= f+q,\quad 0<t\leq T,\quad \mbox{with } u(0)=0,
%\end{aligned}
%\right.
\end{align}
where $T>0$ is a fixed final time, $\gamma>0$ a fixed penalty parameter, $\Delta:H^1_0(\Omega)\cap
H^2(\Omega)\rightarrow L^2(\Omega)$ the Dirichlet Laplacian, $f:(0,T)\rightarrow L^2(\Omega)$ a given
source term, and $u_d:(0,T)\rightarrow L^2(\Omega)$ the target function.
The admissible set $\Uad$ for the control $q$ is defined by
\begin{equation*}
    \Uad=\{ q\in L^2(0,T;L^2\II):~~a\le q \le b ~~\text{a.e. in}~~ \Omega\times(0,T)\},
\end{equation*}
with $a,b\in\mathbb{R}$ and $a< b$. The notation $\Dal u$ in \eqref{eqn:fde} denotes the left-sided Riemann-Liouville
fractional derivative in time $t$ of order $\alpha\in(0,1)$, defined by \cite[p. 70]{KilbasSrivastavaTrujillo:2006}
\begin{equation}\label{eqn:RLderive}
   \Dal u(t)= \frac{1}{\Gamma(1-\alpha)}\frac{\d}{\d t}\int_0^t(t-s)^{-\alpha} u(s)\d s.
\end{equation}
Since $u(0)=0$, the Riemann-Liouville derivative $\Dal u(t) $ coincides  with the usual Caputo derivative
\cite[p. 91]{KilbasSrivastavaTrujillo:2006}.
%The latter allows specifying the initial condition as usual, and below we do not distinguish them.
Further, when $\alpha=1$, $_0\partial_t^\alpha u(t)$ coincides with the first-order derivative $u'(t)$,
and thus the model \eqref{eqn:fde} recovers the standard parabolic problem.

The fractional derivative $\Dal u$ in the model \eqref{eqn:fde} is motivated by an ever-growing list of practical applications
related to subdiffusion processes, in which the mean square displacement grows sublinearly with time $t$, as opposed to linear
growth for normal diffusion. The list includes thermal diffusion in fractal media, protein transport in plasma membrane
and column experiments etc \cite{AdamsGelhar:1992,HatanoHatano:1998,Nigmatulin:1986}.
The numerical analysis of the model \eqref{eqn:fde}
has received much attention. However, the design and analysis of numerical methods for related optimal
control problems only started to attract attention \cite{AntilOtarola:2016,DuWangLiu:2016,LuZuazua:2016,YeXu:2013,YeXu:2015}.
The controllability of \eqref{eqn:fde} was discussed in \cite{FujishiroYammamoto:2014} and \cite{LuZuazua:2016}.
Ye and Xu \cite{YeXu:2013,YeXu:2015} proposed space-time spectral type methods for optimal control problems
under a subdiffusion constraint, and derived error estimates by assuming sufficiently smooth state
and control variables. Antil et al \cite{AntilOtarola:2016} studied an optimal
control problem with space- and time-fractional models, and showed the convergence of the discrete
approximations via a compactness argument. However, no error estimate for the optimal control was
given for the time-fractional case. Zhou and Gong \cite{ZhouGong:2016} proved the
well-posedness of problem \eqref{eqn:ob}--\eqref{eqn:fde} and derived $L^2(0,T;L^2(\Omega))$ error estimates
for the spatially semidiscrete finite element method, and described a time discretization method without error
estimate. To the best of our knowledge, there is no error estimate for time discretizations of
\eqref{eqn:ob}--\eqref{eqn:fde}. It is the main goal of this work to fill this gap.

This work is devoted to the error analysis of both time and space discretizations of
\eqref{eqn:ob}-\eqref{eqn:fde}. The model \eqref{eqn:fde} is discretized by the continuous piecewise
linear Galerkin FEM in space and the L1 approximation \cite{LinXu:2007} or backward Euler convolution quadrature
\cite{Lubich:1986} in time, and the control $q$ by a variational type discretization in
\cite{Hinze:2005}. The analysis relies crucially on $\ell^p(L^2(\Omega))$ error estimates for the fully discrete
finite element solutions of the direct problem with nonsmooth source term. Such results are still unavailable in
the existing literature. We shall derive such estimates in Theorems \ref{thm:err-time}
and \ref{thm:err-time2}, and use them to derive an $O(\tau^{\min({1}/{2}+\alpha-\epsilon,1)}+h^2)$ error
estimate in the discrete $L^2(0,T;L^2(\Omega))$ norm for the numerical solutions of problem
\eqref{eqn:ob}--\eqref{eqn:fde}, where $h$ and $\tau$ denote the mesh size and time stepsize, respectively,
and $\epsilon>0$ is small, cf. Theorems \ref{thm:err-space} and \ref{thm:full}. The $O(\tau^{\min(1/2+\alpha-\epsilon,1)})$
rate contrasts with the $O(\tau)$ rate for the parabolic counterpart
(see, e.g., \cite{MeidnerVexler:2008b,ChrysafinosKaratzas:2014,GongHinzeZhou:2014}).
The lower rate for $\alpha\leq1/2$ is due to the limited smoothing property of
problem \eqref{eqn:fde}, cf. Theorem \ref{thm:reg}. This also constitutes
the main technical challenge in the analysis. Based on the error estimate in the discrete $L^2(0,T;L^2
(\Omega))$ norm, we further derive a pointwise-in-time error estimate $O(\tau^{\alpha-\epsilon}+\ell_h^2h^2)$
(with $\ell_h=\log(2+1/h)$, cf. Theorems \ref{thm:err-space-inf} and \ref{thm:full-Linfty}). Our analysis
relies essentially on the maximal $L^p$-regularity of fractional evolution equations and its discrete
analogue \cite{Bajlekov:2001,JinLiZhou:max-reg}. Numerical experiments are provided to support the theoretical analysis.

The rest of the paper is organized as follows. In Section \ref{sec:prelim}, we discuss the solution regularity
and numerical approximation for problem \eqref{eqn:fde}. In Section \ref{sec:error}, we prove error bounds on
fully discrete approximations to problem \eqref{eqn:ob}--\eqref{eqn:fde}. Finally in Section \ref{sec:numer},
we provide numerical experiments to support the theoretical results. Throughout, the notation $c$ denotes a generic
constant which may differ at each occurrence, but it is always independent of the mesh size $h$ and time stepsize $\tau$.

\section{Regularity theory and numerical approximation of the direct problem}\label{sec:prelim}
In this section, we recall preliminaries and present analysis for the direct problem
\begin{align}\label{eqn:fde-forward}
\Dal u-\Delta u= g, \quad 0<t\leq T,
\quad\mbox{with}\quad
u(0)=0 ,
\end{align}
and its adjoint problem
\begin{align}\label{eqn:fde-adjoint}
\dDal z-\Delta z= \eta, \quad 0\leq t<T,
\quad\mbox{with}\quad
z(T)=0,
\end{align}
where the fractional derivative $ \dDal z$ is defined in \eqref{eqn:fracderiv} below.
In the case $\alpha\in(0,1/2]$, the initial condition should be understood properly: for a rough source  term $g$, the temporal
trace may not exist and the initial condition should be interpreted in a weak sense \cite{GorenfloLuchkoYamamoto:2015}.
Thus we refrain from the case of nonzero initial condition, and leave it to a future work.

\subsection{Sobolev spaces of functions vanishing at $t=0$}
We shall use extensively Bochner-Sobolev spaces $W^{s,p}(0,T;L^2(\Omega))$. For any $s\ge 0$ and
$1\le p< \infty$, we denote by $W^{s,p}(0,T;L^2(\Omega))$ the space of functions $v:(0,T)\rightarrow L^2(\Omega)$,
with the norm defined by interpolation. Equivalently, the space is equipped with the quotient norm
\begin{align}\label{quotient-norm}
\|v\|_{W^{s,p}(0,T;L^2(\Omega))}:=
\inf_{\widetilde v}\|\widetilde v\|_{W^{s,p}({\mathbb R};L^2(\Omega))} ,
\end{align}
where the infimum is taken over all possible extensions $\widetilde v$ that extend $v$ from $(0,T)$ to ${\mathbb R}$.
For any $0<s< 1$, one can define Sobolev--Slobodecki\v{\i} seminorm $|\cdot|_{W^{s,p}(0,T;L^2(\Omega))}$ by
\begin{equation}\label{eqn:SS-seminorm}
   | v  |_{W^{s,p}(0,T;L^2(\Omega))}^p := \int_0^T\int_0^T \frac{\|v(t)-v(\xi)\|_{L^2(\Omega)}^p}{|t-\xi|^{1+ps}} \,\d t\d\xi ,
\end{equation}
and the full norm $\|\cdot\|_{W^{s,p}(0,T;L^2(\Omega))}$ by
\begin{equation*}
\|v\|_{W^{s,p}(0,T;L^2(\Omega))}^p = \|v\|_{L^p(0,T;L^2(\Omega))}^p+|v|_{W^{s,p}
(0,T;L^2(\Omega))}^p .
\end{equation*}
For $s>1$, one can define similar seminorms and norms. Let
$$
C^\infty_L(0,T;L^2(\Omega)):=\{v=w|_{(0,T)}: w\in C^\infty({\mathbb R};L^2(\Omega)): \mbox{supp}(w)\subset [0,\infty)\},
$$
and denote by $W_L^{s,p}(0,T;L^2(\Omega))$ the closure of $C^\infty_L(0,T;L^2(\Omega))$ in $W^{s,p}(0,T;L^2(\Omega))$,
and by $W^{s,p}_R(0,T;L^2(\Omega))$ the closure of $C^\infty_R(0,T;L^2(\Omega))$ in $W^{s,p}(0,T;L^2(\Omega))$, with
$$
C^\infty_R(0,T;L^2(\Omega)):=\{v=w|_{(0,T)}: w\in C^\infty({\mathbb R};L^2(\Omega)): \mbox{supp}(w)\subset (-\infty,T]\}.
$$
By Sobolev embedding, for $v\in W_L^{s,p}(0,T;L^2(\Omega))$, there holds $v^{(j)}(0)=0$ for $j=0,\ldots,[s]-1$ (with $[s]$ being
the integral part of $s>0$), and also $v^{(j)}=0$ if $(s-[s])p>1$. For $v\in W_L^{s,p}(0,T;L^2(\Omega))$, the zero extension
of $v$ to the left belongs to $ W^{s,p}(-\infty,T;L^2(\Omega))$, and $W_L^{s,p}(0,T;L^2(\Omega))=W^{s,p}(0,T;L^2(\Omega))$, if
$s<1/p$. We abbreviate $W^{s,2}_L(0,T;L^2(\Omega))$ as $H^s_L(0,T;L^2(\Omega))$, and likewise $H^s_R(0,T;L^2(\Omega))$ for
$W^{s,2}_R(0,T;L^2(\Omega))$.

Similar to the left-sided fractional derivative $\Dal u$ in \eqref{eqn:RLderive}, the right-sided Riemann-Liouville fractional
derivative $\dDal v(t)$ in \eqref{eqn:fde-adjoint} is defined by
\begin{align}\label{eqn:fracderiv}
\dDal v(t):= - \frac{1}{\Gamma(1-\al)} \frac{\d}{\d t}\int_t^T(s-t)^{-\al}v(s)\, \d s .
\end{align}
Let any $p\in(1,\infty)$ and $p'\in(1,\infty)$ be conjugate to each other, i.e., $1/p+1/p'=1$.
Since for $u\in W^{\alpha,p}_L(0,T;L^2(\Omega)),v\in W^{\alpha,p'}_R(0,T;L^2(\Omega))$, we have
$_0\partial_t^\alpha u\in L^p(0,T;L^2(\Omega)), {_t\partial_T^\alpha} v\in L^{p'}(0,T;L^2(\Omega))$.
Thus, there holds \cite[p. 76, Lemma 2.7]{KilbasSrivastavaTrujillo:2006}:
\begin{equation}\label{eqn:adjoint}
\int_0^T ({_0\partial_t^\alpha u(t)}) v(t) \d t = \int_0^T u(t) ({_t\partial_T^\alpha v(t)}) \d t ,
\ \ \forall u \in W^{\alpha,p}_L(0,T;L^2(\Omega)),\, v \in  W^{\alpha,p'}_R(0,T;L^2(\Omega)).
\end{equation}

\subsection{Regularity of the direct problem}
The next maximal $L^p$-regularity holds \cite{Bajlekov:2001}, and an analogous result holds for
 \eqref{eqn:fde-adjoint}.
\begin{lemma} \label{lem:max-lp}
If $u_0=0$ and $f\in L^p(0,T;L^2(\Omega))$ with $1<p<\infty$, then \eqref{eqn:fde-forward}
has a unique solution $u\in L^p(0,T; H_0^1(\Omega)\cap H^2(\Omega))$ such that $\Dal u\in L^p(0,T;L^2(\Omega))$ such that
\begin{equation*}
\|u\|_{L^p(0,T; H^2(\Omega))}
+\|\Dal u\|_{L^p(0,T;L^2(\Omega))}
\le c\|f\|_{L^p(0,T;L^2(\Omega))},
\end{equation*}
where the constant $c$ is independent of $f$ and $T$.
\end{lemma}

Now we give a regularity result. % (cf. \cite[Theorem 4.1]{GorenfloLuchkoYamamoto:2015} for related results).
\begin{theorem}\label{thm:reg}
For $g\in {W^{s,p}}( {0,T;{L^2}(\Omega )})$,  $s\in [0,{1}/{p})$ and $p\in(1,\infty)$, problem \eqref{eqn:fde-forward} has a unique solution
$u\in W^{\alpha+s,p}( {0,T;{L^2}(\Omega )}) \cap W^{s,p}(0,T;H_0^1(\Omega) \cap H^2(\Omega))$, which satisfies
\begin{equation*}
  \|u\|_{W^{\alpha+s,p} (0,T;L^2(\Omega))} + \|u\|_{W^{s,p}(0,T;H^2(\Omega))}\leq c\|g\|_{W^{s,p} (0,T;L^2(\Omega))}.
\end{equation*}
Similarly, for $\eta\in W^{s,p}(0,T;L^2(\Omega ))$,  $s\in [0,{1}/{p})$ and $p\in(1,\infty)$, problem \eqref{eqn:fde-adjoint} has
a unique solution $z\in W^{\alpha+s,p}(0,T;L^2(\Omega )) \cap W^{s,p}(0,T;H_0^1(\Omega) \cap H^2(\Omega))$, which satisfies
\begin{equation*}
  \|z\|_{W^{\alpha+s,p} (0,T;L^2(\Omega))} + \|z\|_{W^{s,p}(0,T;H^2(\Omega))}\leq c\|\eta\|_{W^{s,p} (0,T;L^2(\Omega))}.
\end{equation*}
\end{theorem}
\begin{proof}
For $g\in W^{s,p}(0,T;L^2(\Omega))$,  $s\in [0,{1}/{p})$ and $p\in(1,\infty)$,
extending $g$ to be zero on $\Omega\times[(-\infty,0)\cup(T,\infty)]$ yields $g\in  W^{s,p}({\mathbb R};L^2(\Omega))$ and
\begin{equation}\label{eqn:Wsp-extension}
  \|g\|_{W^{s,p} ({\mathbb R};L^2(\Omega))}
  \le c\|g\|_{W^{s,p} (0,T;L^2(\Omega))} .
\end{equation}
For the zero extension to the left, we have the identity
$_0\partial_t^\alpha g(t) = {_{-\infty}\partial_t^\alpha} g(t)$ for $t\in[0,T],$
and there holds the relation $\widehat {_{-\infty}\partial_t^\alpha g} = (i\xi)^\alpha\widehat g(\xi)$
\cite[p. 90]{KilbasSrivastavaTrujillo:2006}, where ~$\widehat{~}$~ denotes taking Fourier transform in
 $t$, and $\widehat g$ the Fourier transform of $g$. Then, with $^\vee$ being the inverse Fourier
transform in $\xi$, $u= [((i\xi)^\alpha-\Delta)^{-1}\widehat g(\xi)]^\vee$ is a solution
of \eqref{eqn:fde-forward} and
\begin{equation*}
(1+|\xi|^2)^{\frac{\alpha+s}{2}}\widehat u(\xi)= (1+|\xi|^2)^{\frac{\alpha}{2}}((i\xi)^\alpha-\Delta)^{-1}(1+|\xi|^2)^{\frac{s}{2}}\widehat g(\xi) .
\end{equation*}

Let $\Delta$ be the Dirichlet Laplacian, with domain $D(\Delta) = H_0^1(\Omega)\cap H^2(\Omega)$.
Then the self-adjoint operator is invertible from
$L^2(\Omega)$ to $D(\Delta)$, and generates a bounded analytic semigroup \cite[Example 3.7.5]{ABHN}. Thus the operator
\begin{align}\label{BD-multp-1}
(1+|\xi|^2)^{\frac{\alpha}{2}}((i\xi)^\alpha-\Delta)^{-1}
\end{align}
is bounded from $L^2(\Omega)$ to $D(\Delta)$ in a small neighborhood $\mathcal{N}$ of $\xi=0$. Further, in $\mathcal{N}$, the operator
\begin{align}\label{BD-multp-2}
\xi\frac{\d}{\d\xi}(1+|\xi|^2)^{\frac{\alpha}{2}}((i\xi)^\alpha-\Delta)^{-1}
=& \frac{\alpha |\xi|^2}{1+|\xi|^2} (1+|\xi|^2)^{\frac{\alpha}{2}}((i\xi)^\alpha-\Delta)^{-1} \nonumber \\
&+\alpha (1+|\xi|^2)^{\frac{\alpha}{2}}((i\xi)^\alpha-\Delta)^{-1}
 (i\xi)^{\alpha}((i\xi)^\alpha-\Delta)^{-1}
\end{align}
is also bounded. If $\xi$ is away from $0$, then
\begin{equation*}
(1+|\xi|^2)^{\frac{\alpha}{2}}((i\xi)^\alpha-\Delta)^{-1}= (i\xi)^{-\alpha}(1+|\xi|^2)^{\frac{\alpha}{2}}(i\xi)^{\alpha}((i\xi)^\alpha-\Delta)^{-1} ,
\end{equation*}
and in the latter case the inequalities
\begin{equation*}
|(i\xi)^{-\alpha}(1+|\xi|^2)^{\frac{\alpha}{2}}|\leq c \quad \mbox{and}\quad \|(i\xi)^{\alpha}((i\xi)^\alpha-\Delta)^{-1}\|\leq c
\end{equation*}
imply the boundedness of \eqref{BD-multp-1} and \eqref{BD-multp-2}.

Since boundedness of operators is equivalent to $R$-boundedness of operators in $L^2(\Omega)$ (see \cite{KunstmannWeis:2004}
for the concept of $R$-boundedness), the boundedness of \eqref{BD-multp-1} and \eqref{BD-multp-2} implies that
\eqref{BD-multp-1} is an operator-valued Fourier multiplier \cite[Theorem 3.4]{Weis:2001}, and thus
\begin{equation*}
\begin{aligned}
\|u\|_{W^{\alpha+s,p}({\mathbb R};L^2(\Omega))}
&\le \|[(1+|\xi|^2)^{\frac{\alpha+s}{2}}\widehat u(\xi)]^\vee\|_{L^p({\mathbb R};L^2(\Omega))} \\
&=\|[(1+|\xi|^2)^{\frac{\alpha}{2}}((i\xi)^\alpha-\Delta)^{-1}(1+|\xi|^2)^{\frac{s}{2}}\widehat g(\xi)]^\vee\|_{L^p({\mathbb R};L^2(\Omega))} \\
&\le c\|[(1+|\xi|^2)^{\frac{s}{2}}\widehat g(\xi)]^\vee\|_{L^p({\mathbb R};L^2(\Omega))} \le c\|g\|_{W^{s,p}({\mathbb R};L^2(\Omega))} .
\end{aligned}
\end{equation*}
This and  \eqref{eqn:Wsp-extension} implies the desired bound on $\|u\|_{W^{\alpha+s,p}(0,T;L^2(\Omega))}$.
The estimate
\begin{equation*}
  \|u\|_{W^{s,p}(0,T;H^2(\Omega))}\le c\|g\|_{W^{s,p}(0,T;L^2(\Omega))}
\end{equation*}
follows similarly by replacing $(1+|\xi|^2)^{\frac{\alpha}{2}}((i\xi)^\alpha-\Delta)^{-1}$ with
$\Delta((i\xi)^\alpha-\Delta)^{-1}$ in the proof.
\end{proof}

\begin{remark}\label{rem:ext-reg}
Below we only use the cases ``$p=2$, $s=\min(1/2-\epsilon,\alpha-\epsilon)$'' and ``$p>\max({1}/{\alpha},{1}/(1-\alpha))$,
$s={1}/{p}-\epsilon$'' of Theorem \ref{thm:reg}. Both cases satisfy the conditions of Theorem \ref{thm:reg}.
A similar assertion holds for the more general case $g\in W_L^{s,p}(0,T;L^2(\Omega))$, $s>0$ and $1<p<\infty$:
\begin{equation*}
  \|u\|_{W^{\alpha+s,p} (0,T;L^2(\Omega))} + \|u\|_{W^{s,p}(0,T;H^2(\Omega))}\leq c\|g\|_{W_L^{s,p} (0,T;L^2(\Omega))}.
\end{equation*}
In fact, for $g\in W_L^{s,p}(0,T;L^2(\Omega))$, the zero extension of $g$ to $t\leq 0$ belongs to $W^{s,p}(-\infty,T;L^2(\Omega))$,
which can further be boundedly extended to a function in $W^{s,p}(\mathbb{R};L^2(\Omega))$. Then the argument in
Theorem \ref{thm:reg} gives the desired assertion. This also indicates a certain compatibility condition for regularity pickup.
\end{remark}

\subsection{Numerical scheme}\label{sec:error-direct}
Now we describe numerical treatment of the forward problem \eqref{eqn:fde-forward}, which forms the basis for the
fully discrete scheme of problem \eqref{eqn:ob}--\eqref{eqn:fde} in Section
\ref{sec:error}. % Error estimates of the approximations will be given in Section \ref{ssec:error}.
We denote by $\mathcal{T}_h$ a shape-regular and quasi-uniform triangulation of the
domain $\Omega $ into $d$-dimensional simplexes, and let
\begin{equation*}
  X_h= \left\{v_h\in H_0^1(\Omega):\ v_h|_K \mbox{ is a linear function},\ \forall\, K \in \mathcal{T}_h\right\}
\end{equation*}
be the finite element space consisting of continuous piecewise linear functions.
The $L^2(\Omega)$-orthogonal projection $P_h:L^2(\Omega)\to X_h$ is defined by
$ (P_h \fy,\chi_h)  =(\fy,\chi_h)$, for all $\fy\in L^2(\Omega), \chi_h\in X_h,$
where $(\cdot,\cdot)$ denotes the $L^2\II$ inner product. Then the spatially semidiscrete Galerkin FEM for
 problem \eqref{eqn:fde-forward} is to find $u_h(t)\in X_h$ such that $u_h(0)=0$ and
\begin{equation}\label{eqn:fem}
(\Dal u_h(t),\chi_h) + (\nabla u_h(t),\nabla \chi_h) = (g(t),\chi_h),\quad\forall \chi_h\in X_h,\,\,\,\forall\, t\in(0,T], \\
\end{equation}
By introducing the discrete Laplacian $\Delta_h: X_h\to X_h$, defined by
$  -(\Delta_h\fy_h,\chi_h)=(\nabla\fy_h,\nabla\chi_h),$ for all $\fy_h,\,\chi_h\in X_h,$
problem \eqref{eqn:fem} can be written as
\begin{equation}\label{eqn:fdesemidis}
\Dal u_h(t) -\Delta_h u_h(t) =P_hg(t), \quad\forall\, t\in(0,T] ,\quad \mbox{with }u_h(0)=0.
\end{equation}
Similar to Theorem \ref{thm:reg},  for $s<{1}/{p}$ there holds
\begin{equation}\label{eqn:stability-semi}
  \|u_h\|_{W^{\alpha+s,p}(0,T;L^2(\Omega))}+ \|\Delta_h u_h\|_{W^{s,p}(0,T;L^2(\Omega))} \leq c \|g\|_{W^{s,p}(0,T;L^2(\Omega))} ,
\end{equation}
where the constant $c$ is independent of $h$ (following the proof of Theorem \ref{thm:reg}). Lemma \ref{lem:max-lp}
and Remark \ref{rem:ext-reg} remain valid for the semidiscrete solution $u_h$, e.g.,
\begin{equation}\label{eqn:max-lp-h}
\|\Delta_hu_h\|_{L^p(0,T;L^2(\Omega))}
+\|\Dal u_h\|_{L^p(0,T;L^2(\Omega))}
\le c\|f\|_{L^p(0,T;L^2(\Omega))}.
\end{equation}
These assertions will be used extensively below without explicitly referencing.

To discretize \eqref{eqn:fdesemidis} in time, we uniformly partition $[0,T]$
with grid points $t_n=n\tau$, $n,=0,1,2,\dots,N$ and a time stepsize $\tau=T/N\leq1$, and approximate
$\Dal\varphi(t_n)$ by (with $\varphi^j=\varphi(t_j)$):
\begin{equation}\label{eqn:timestepping}
   \bPtau \varphi^n  = \tau^{-\alpha}\sum_{j=0}^n\beta_{n-j}\varphi^j,
\end{equation}
where $\beta_j$ are suitable weights. We consider
two methods: L1 scheme \cite{LinXu:2007} and backward Euler convolution quadrature (BE-CQ)
\cite{Lubich:1986}, for which $\beta_j$ are respectively given by
\begin{align*}
&\mbox{L1 scheme:}
&&\beta_0=1, \quad \mbox{and } \beta_j=(j+1)^{1-\alpha}-2j^{1-\alpha}+(j-1)^{1-\alpha},\ \ j=1,2,\ldots,N,\\
&\mbox{BE-CQ:}&&\beta_0 = 1, \quad \mbox{and } \beta_j = -\beta_{j-1}(\alpha-j+1)/j,\ \ j = 1,2,\ldots,N.
\end{align*}
Both schemes extend the classical backward Euler scheme to the fractional case.
Then we discretize problem \eqref{eqn:fde-forward} by: with $g_h^n=P_hg(t_n)$, find $U_h^n\in X_h$ such that
\begin{equation}\label{eqn:fully-variant}
    \bPtau  U_h^n -\Delta_h  U_h^n  = g_h^n,\quad n=1,2,\ldots,N,\quad \mbox{with }  U_h^0 = 0.
\end{equation}
By \cite[Section 5]{JinLiZhou:nonlinear} and \cite[Theorem 3.6]{JinLazarovZhou:SISC2016}, we have the following error bound.

\begin{lemma}\label{lem:fully-variant}
For $g\in W^{1,p}(0,T;L^2(\Omega))$, $1\leq p\leq\infty$, let $u_h$ and $U_h^n$ be the semidiscrete solution and
 fully discrete solution, respectively, in \eqref{eqn:fdesemidis} and \eqref{eqn:fully-variant}. Then there holds
\begin{equation*}
   \| U_h^n - u_h(t_n)   \|_{L^2\II} \le c\tau t_n^{\alpha-1}\|g(0)\|_{L^2(\Omega)}+c\tau \int_0^{t_n} (t_{n+1}-s)^{\al-1} \|g'(s) \|_{L^2\II} \d s.
\end{equation*}
\end{lemma}

\begin{remark}
Lemma \ref{lem:fully-variant} slightly refines the estimates in \cite{JinLiZhou:nonlinear,JinLazarovZhou:SISC2016},
but can be proved in the same way using the following estimates in the proof of \cite[Section 5]{JinLiZhou:nonlinear}:
\begin{equation*}
   t_n^{\alpha-1}\le c (t+\tau)^{\alpha-1} \ \
\text{and}
  \ \  \int_t^{t_n} s^{\alpha-1} \d s \le c (t+\tau)^{\alpha-1} ,
\,\,\,
\mbox{ for\, $t\in[t_{n-1},t_n]$\, and \,$n=1,\ldots,N$.}
\end{equation*}
\end{remark}

For any Banach space $X$, we define
$$
\|(U_h^n)_{n=0}^N\|_{\ell^p(X)}:=
\left\{
\begin{aligned}
&\bigg(\sum_{n=0}^N\tau\|U_h^n\|_{X}^p\bigg)^{{1}/{p}}  &&\mbox{if}\,\,\, 1\le p<\infty,\\
&\max_{0\le n\le N}\|U_h^n\|_{X} &&\mbox{if}\,\,\, p=\infty .
\end{aligned}\right.
$$
Then the maximal $\ell^p$-regularity estimate holds for \eqref{eqn:fully-variant} \cite[Theorems 5 and 7]{JinLiZhou:max-reg}.
\begin{lemma}\label{lem:max-reg}
The solutions $(U_h^n)_{n=1}^N$ of \eqref{eqn:fully-variant} satisfy the following estimate:
\begin{align*}
\|(\bPtau U_h^n)_{n=1}^N\|_{\ell^p(L^2(\Omega))} + \|(\Delta_h U_h^n)_{n=1}^N\|_{\ell^p(L^2(\Omega))}
\leq c_{p}\|(g_h^n)_{n=1}^N\|_{\ell^p(L^2(\Omega))},
\quad\forall\, 1<p<\infty.
\end{align*}
\end{lemma}

\subsection{Error estimates}\label{ssec:error}
Now we present $\ell^p(L^2\II)$ error estimates for $g\in W^{s,p}(0,T;L^2(\Omega))$,
$0\leq s\leq 1$, $1\leq p\leq \infty$. Error analysis for
such $g$ is unavailable in the literature. First, we give an interpolation error estimate.
This result seems standard, but we are unable to find a proof, and thus
include a proof in Appendix \ref{app:interp}.

\begin{lemma}\label{lem:L-interp}
For $v\in W^{s,p}(0,T;L^2(\Omega))$, $1<p<\infty$ and $s\in(1/p,1]$, let $\bar v^n = \tau^{-1}\int_{t_{n-1}}^{t_n} v(t)\,\d t$,
there holds
\begin{equation*} %\label{eq-lem:L-interp}
 \| (v(t_n) - \bar v^n)_{n=1}^N  \|_{\ell^p(L^2\II)}   \le c \tau^{s}    \| v \|_{W^{s,p}(0,T;L^2(\Omega))} .
\end{equation*}\vspace{-10pt}
\end{lemma}

Our first result is an error estimate for $g\in W_L^{s,p} (0,T;L^2\II)$ (i.e., compatible
source). Since $g$ may not be smooth enough in time for pointwise evaluation,
we define the averages $\bar g_h^n=\tau^{-1}\int_{t_{n-1}}^{t_n}P_h g(s)\d s$, and
consider a variant of the scheme \eqref{eqn:fully-variant} for problem \eqref{eqn:fde-forward}:
find $\overline U_h^n \in X_h$ such that
\begin{equation}\label{eqn:fully}
\bPtau \overline U_h^n -\Delta_h \overline U_h^n = \bar g_h^n ,\quad n = 1,\ldots,N, \quad \mbox{with }
\overline U_h^0=0.
\end{equation}

\begin{theorem}\label{thm:err-time}
For $g\in W^{s,p}_L(0,T; L^2\II)$, $1  < p<\infty$ and $s\in[0,1]$, let $u_h$ and $\overline U_h^n$ be the solutions of problems \eqref{eqn:fdesemidis}
and \eqref{eqn:fully}, respectively, and $\bar u_h^n:=\tau^{-1}\int_{t_{n-1}}^{t_n}u_h(s)\d s$. Then there holds
\begin{align*}
 \| (\overline U_h^n - \bar u_h^n)_{n=1}^N  \|_{\ell^p(L^2\II)}   \le c\tau^{ s} \| g \|_{W^{s,p}_L(0,T;L^2\II)}.
\end{align*} \vspace{-10pt}
\end{theorem}
\begin{proof}
By H\"{o}lder's inequality and \eqref{eqn:stability-semi} (with $s=0$), we have
\begin{equation*}
  \begin{aligned}
  \tau\sum_{n=1}^N\|\bar u_h^n\|_{L^2\II}^p & = \tau \sum_{n=1}^N \|\tau^{-1}\int_{t_{n-1}}^{t_n}u_h(s)\d s\|_{L^2\II}^p
  \leq \tau^{1-p}\sum_{n=1}^N\Big(\int_{t_{n-1}}^{t_n}\|u_h(s)\|_{L^2(\Omega)}\d s\Big)^p\\
  &\leq \int_0^T\|u_h(s)\|_{L^2\II}^p\d s\leq c\|g\|_{L^p(0,T;L^2\II)}^p .
  \end{aligned}
\end{equation*}
Similarly, by applying Lemma \ref{lem:max-reg} to \eqref{eqn:fully} and the $L^2(\Omega)$ stability of $P_h$, we have
\begin{equation*}
  \|(\overline U_h^n)_{n=1}^N\|_{\ell^p(L^2(\Omega))} \le c \|  (\bar g_h^n)_{n=1}^N  \|_{\ell^p(L^2(\Omega))}
  \le c \| g \|_{L^p(0,T;L^2\II)} .
\end{equation*}
This and the triangle inequality show the assertion for $s=0$.

Next we consider $g\in W_L^{1,p}(0,T;L^2\II)$, and resort to \eqref{eqn:fully-variant}. Since $g(0)=0$, by Lemma
\ref{lem:fully-variant},
\begin{equation*}
   \| U_h^n - u_h(t_n)   \|_{L^2\II} \le c\tau \int_0^{t_n} (t_n+\tau-s)^{\al-1} \| g'(s) \|_{L^2\II} \d s.
\end{equation*}
This directly implies
\begin{align}
   \| ( U_h^n -u_h(t_n))_{n=1}^N  \|_{\ell^\infty(L^2\II)}
    & \le  c \tau  \| g \|_{W_L^{1,\infty}(0,T;L^2\II)}.\label{eqn:err-U-u-inf}
\end{align}
Further, let $\psi(s)=\tau\sum_{n=1}^N(t_n+\tau-s)^{\alpha-1}\chi_{[0,t_n]}$, where $\chi_S$ denotes the
characteristic function of a set $S$. Then clearly, we have
\begin{equation*}
  \sup_{s\in [0,T]}\psi(s)=\psi(T)=\tau\sum_{n=1}^N(n\tau)^{\alpha-1} \leq \int_0^Ts^{\alpha-1}\d s = \alpha^{-1}T^{\alpha}\leq c_T .
\end{equation*}
Therefore,
\begin{align}
 \| (U_h^n -u_h(t_n))_{n=1}^N  \|_{\ell^1(L^2\II)}
&\le c\tau^2 \sum_{n=1}^N \int_0^{t_n} (t_n+\tau-s)^{\al-1} \| g'(s) \|_{L^2\II} \,\d s\nonumber\\
&= c\tau\int_0^T\psi(s)\|g'(s)\|_{L^2(\Omega)}\d s \le c_T\tau   \| g \|_{W^{1,1}(0,T;L^2\II)}.\label{eqn:err-U-u-1}
\end{align}
Then \eqref{eqn:err-U-u-inf}, \eqref{eqn:err-U-u-1} and Riesz-Thorin
interpolation theorem \cite[Theorem 1.1.1]{bergh:2012}, yield for any $1<p<\infty$
\begin{equation*}%\label{Unh-uhtn1}
 \| (U_h^n -u_h(t_n))_{n=1}^N  \|_{\ell^p(L^2\II)}
  \le  c \tau  \| g \|_{W_L^{1,p}(0,T;L^2\II)}.
\end{equation*}
Since $U_h^n-\overline U_h^n$ satisfies the discrete scheme, cf. \eqref{eqn:fully-variant}
and \eqref{eqn:fully}, Lemmas \ref{lem:max-reg} and \ref{lem:L-interp} imply
\begin{equation*}%\label{eqn:u-pert}
 \|(U_h^n - \overline U_h^n)_{n=1}^N\|_{\ell^p(L^2(\Omega))} \le  c  \| (\bar g_h^n - P_hg(t_n))_{n=1}^N\|_{\ell^p(L^2(\Omega))} \leq c\tau\|g\|_{W^{1,p}(0,T;L^2(\Omega))}.
\end{equation*}
Further, by Lemma \ref{lem:L-interp} and Remark \ref{rem:ext-reg}, we have
\begin{equation*}%\label{Unh-uhtn3}
\|  (u_h(t_n) - \bar u_h^n)_{n=1}^N\|_{\ell^p(L^2\II)} \leq c\tau \|u_h\|_{W^{1,p}(0,T;L^2\II)} \leq c\tau\|g\|_{W_L^{1,p}(0,T;L^2\II) }.
\end{equation*}
The last three estimates show the assertion for $s=1$.
The case $0<s<1$ follows by interpolation.
\end{proof}

In Theorem \ref{thm:err-time}, we compare the numerical solution $\overline {U}_h^n$ to \eqref{eqn:fully} with the time-averaged
solution $\bar u_h^n$, instead of $u_h(t_n)$. This is due to possible insufficient temporal regularity of $u_h$:
it is unclear how to define $u_h(t_n)$ for $t_n\in(0,T]$ for $g\in W_L^{s,p}(0,T; L^2\II)$ with $s+\alpha<1/p$. For $s\in(1/p,1]$,
$W_L^{s,p}(0,T; L^2\II) \hookrightarrow C([0,T] ;L^2\II)$ and so Theorem \ref{thm:err-time} requires the condition
$g(0)=0$. Such a compatibility condition at $t=0$ is not necessarily satisfied by \eqref{eqn:ob}--\eqref{eqn:fde}. Hence, we state an error estimate
below for a smooth but incompatible source $g\in W^{s,p}(0,T; L^2\II)$.

\begin{theorem}\label{thm:err-time2}
For $g\in W^{s,p}(0,T; L^2\II)$ with $p\in(1,\infty)$ and $s\in(1/p,1)$, let $u_h$ and $U_h^n$ be the
solutions of \eqref{eqn:fdesemidis} and \eqref{eqn:fully-variant}, respectively. Then there holds
\begin{align}
\| (U_h^n -  u_h(t_n))_{n=1}^N  \|_{\ell^p(L^2\II)}    \le c \tau^{\min(1/p+\alpha,s)} \| g \|_{W^{s,p}(0,T; L^2\II)} .\label{l2L2Uhn-0}
\end{align}
Moreover, if $p>{1}/{\alpha}$ is so large that $\alpha\in(0,{1}/{p'})$, then
\begin{align}
\| (U_h^n -  u_h(t_n))_{n=1}^N  \|_{\ell^\infty(L^2\II)}    \le c \tau^{\alpha} \| g \|_{W^{1/p+\alpha,p}(0,T; L^2\II)}. \label{l2L2Uhn-00}
\end{align}
\end{theorem}
\begin{proof}
For $g(x,t)\equiv g(x)$, which belongs to $W^{1,p}(0,T;L^2(\Omega))$, by Lemma
\ref{lem:fully-variant} we have
\begin{align} \label{l2L2Uhn-0-}
  \|(U_h^n - u_h(t_n))_{n=1}^N \|_{\ell^p(L^2\II)}^p
  & =  \tau \sum_{n=1}^N  \| U_h^n - u_h(t_n)  \|_{L^2\II}^p
  \le  c \tau^{p+1}  \|  g \|_{L^2\II}^p \sum_{n=1}^N  t_n^{p(\alpha-1)} \nonumber  \\
  & \le c \tau^{p\alpha+1} \|  g \|_{L^2\II}^p  + c \tau^{p}  \|  g \|_{L^2\II}^p \int_\tau^T  t^{p(\alpha-1)} \,\d t.\nonumber
\end{align}
Now for $s<1$, there holds
\begin{equation*}
  \int_\tau^Tt^{p(\alpha-1)} \,\d t
  \le  \left\{\begin{aligned}
        &c \tau^{p\alpha+1} &&\mbox{if}\,\,\,\alpha \in (0,1/p') \\
        &c \tau^{p} &&\mbox{if}\,\,\,\alpha \in (1/p' ,1)\\
        &c \tau^p \Big(1+\ln (T/\tau)\Big) &&\mbox{if}\,\,\,\alpha = 1/p'
                   \end{aligned}
          \right.
     \le  \left\{\begin{aligned}
        &c \tau^{p(1/p+\alpha)}  &&\mbox{if}\,\,\,\alpha \in (0,1/p') \\
        &c \tau^{ps}  &&\mbox{if}\,\,\,\alpha \in (1/p' ,1)\\
        &c \tau^{ps} &&\mbox{if}\,\,\,\alpha = 1/p'
                   \end{aligned}
          \right.
\end{equation*}
which together with the preceding estimate implies
\begin{equation*}%\label{l2L2Uhn-1}
  \|(U_h^n - u_h(t_n))_{n=1}^N \|_{\ell^p(L^2(\Omega))} \le c\tau^{\min(1/p+\alpha,s)} \| g \|_{W^{s,p}(0,T; L^2(\Omega))} .
\end{equation*}
For $g\in W^{s,p}(0,T; L^2\II)$, by Sobolev embedding, $g(0)$ exists,
and in the splitting $g(t)=g(0) + (g(t)-g(0))$, there holds
$
  \|g-g(0)\|_{W_L^{s,p}(0,T;L^2(\Omega))}\leq c\|g\|_{W^{s,p}(0,T;L^2(\Omega))}.
$
Let $v_h$ be the semidiscrete solution for the source $g(t)-g(0)$, and $\bar v_h^n=\int_{t_{n-1}}^{t_n}v_h(t)\d t$.
Since $g(t)-g(0)\in W_L^{s,p}(0,T; L^2\II)$, by Theorem \ref{thm:err-time}, the corresponding fully discrete solution
$\overline V_h^n$ by \eqref{eqn:fully} satisfies
\begin{equation}\label{l2L2Uhn-2}
  \|(\overline V_h^n- \bar v_h^n)_{n=1}^N\|_{\ell^p(L^2\II)}  \leq c\tau^{s}\|g-g(0)\|_{W_L^{s,p}(0,T; L^2\II)}.
\end{equation}
Further, by Lemma \ref{lem:L-interp} and \eqref{eqn:stability-semi}, we have
\begin{align*}
  \|(v_h(t_n)-\bar v_h^n)_{n=1}^N\|_{\ell^p(L^2\II)}
  \leq c \tau^{s} \|v_h\|_{W^{s,p}(0,T;L^2(\Omega))}
   &\leq c\tau^{s} \|g-g(0)\|_{W_L^{s,p}(0,T; L^2\II)}.
\end{align*}
Similarly, for the fully discrete solution $V_h^n$ for the source $g(t)-g(0)$ by \eqref{eqn:fully-variant},
from Lemmas \ref{lem:max-reg} and \ref{lem:L-interp}, we deduce
\begin{equation}\label{eqn:U-barU}
  \begin{aligned}
  \|(V_h^n-\overline V_h^n)_{n=1}^N\|_{\ell^p(L^2(\Omega))}&\leq c \|(P_hg(t_n)-\bar g_h^n)_{n=1}^N\|_{\ell^p(L^2(\Omega))} \\
   &\leq c\tau^{s}  \|g\|_{W^{s,p}(0,T;L^2(\Omega))}.
  \end{aligned}
\end{equation}
These estimates together with the triangle inequality give \eqref{l2L2Uhn-0}.

%To prove \eqref{l2L2Uhn-00}, for the case $g(x,t)\equiv g(x)$, with $\alpha\in (0,{1}/{p'})$, \eqref{l2L2Uhn-0-} implies
%\begin{equation*}
%  \|(U_h^n - u_h(t_n))_{n=1}^N \|_{\ell^p(L^2\II)} \le c\tau^r \| g \|_{W^{r,p}(0,T; L^2\II)},
%\end{equation*}
%with $r=1/p +\alpha$. For a general source $g$, from \eqref{l2L2Uhn-2}, we deduce
%\begin{equation*}
%\|(\overline V_h^n- \bar v_h^n)_{n=1}^N\|_{\ell^p(L^2\II)} \leq c\tau^r\|g-g(0)\|_{W_L^{r,p}(0,T; L^2\II)},
%%\le c\tau^r\|g\|_{W^{r,p}(0,T; L^2\II)},
%\end{equation*}
%and by Lemma \ref{lem:L-interp} and Remark \ref{rem:ext-reg},
%\begin{align*}
%    \|(v_h(t_n)-\bar v_h^n)_{n=1}^N\|_{\ell^p(L^2\II)}\leq  \tau^r\|v_h\|_{W^{r,p}(0,T;L^2(\Omega))}
%   \leq  c\tau^r\|g-g(0)\|_{W_L^{r,p}(0,T; L^2\II)}. % \leq c\tau^r\|g\|_{W^{r,p}(0,T; L^2\II)} .
%\end{align*}
%These two estimates, \eqref{eqn:U-barU} with $s=r$ and the inverse inequality in time give \eqref{l2L2Uhn-00}.

Finally, \eqref{l2L2Uhn-00} follows by the inverse inequality in time and \eqref{l2L2Uhn-0} with $s=1/p+\alpha$, i.e.,
\begin{align*}
\| (U_h^n -  u_h(t_n))_{n=1}^N  \|_{\ell^\infty(L^2\II)}
 &\le c \tau^{-1/p} \| (U_h^n -  u_h(t_n))_{n=1}^N  \|_{\ell^p(L^2\II)} \\
 &\le c\tau^\alpha \|g\|_{W^{1/p+\alpha,p}(0,T; L^2\II)} .
\end{align*}
This completes the proof of the theorem.
\end{proof}

\begin{remark}\label{rmk:err}
The estimate \eqref{l2L2Uhn-0} is not sharp since,
according to the proof, the restriction $s<1$ is only needed for $\alpha = {1}/{p'}$. Nonetheless,
it is sufficient for the error analysis in Section \ref{sec:error}.
\end{remark}
\section{The optimal control problem and its numerical approximation}\label{sec:error}
In this section, we develop a numerical scheme for problem \eqref{eqn:ob}--\eqref{eqn:fde},
and derive error bounds for the spatial and temporal discretizations.

\subsection{The continuous problem}
The first-order optimality condition of  \eqref{eqn:ob}-\eqref{eqn:fde} was given in \cite[Theorem 3.4]{ZhouGong:2016}.
\begin{theorem}\label{thm:nec-opt}
Problem \eqref{eqn:ob}-\eqref{eqn:fde} admits a unique solution $q\in \Uad$.
There exist a state $u\in L^2(0,T; D(\Delta))\cap H_L^\al(0,T; L^2\II)$ and an adjoint
$z\in L^2(0,T; D(\Delta))\cap H_R^\al(0,T; L^2\II)$ such that $(u,z,q)$ satisfies
\begin{align}
   \Dal u-\Delta u&= f+q,\quad  0<t\leq T,\quad \mbox{with}\quad u(0)=0 ,\label{eqn:fde1}\\
   _t\partial_T^\alpha z-\Delta z&= u-u_d,\quad   0\le t<T,\quad \mbox{with}\quad z(T)=0,\label{eqn:fde2}\\
   (\gamma q + z,&v-q)_{L^2(0,T;L^2(\Omega))}\ge 0,\quad \quad \forall v\in \Uad.\label{eqn:fde3}
\end{align}
where $(\cdot,\cdot)_{L^2(0,T;L^2(\Omega))}$ denotes the $L^2(0,T;L^2(\Omega))$ inner product.
\end{theorem}

Let $P_{\Uad}$ be the nonlinear pointwise projection operator defined by
\begin{equation}\label{Def-PUq}
    P_{\Uad} (q) = \max(a,\min(q,b)).
\end{equation}
It is bounded on $W^{s,p}(0,T; L^2(\Omega))$ for $0\leq s\leq 1$ and $1\leq p\leq \infty$
%(see \cite[Lemma 3.3]{KunischVexler:2007} for the proof for the case $p=2$)
\begin{equation}\label{eqn:bdd-Puad}
  \|P_{\Uad}u\|_{W^{s,p}(0,T;L^2(\Omega))}\leq c\|u\|_{W^{s,p}(0,T; L^2(\Omega))}.
\end{equation}
This estimate holds trivially for $s=0$ and $s=1$ (see
\cite[Corollary 2.1.8]{Ziemer:1989}), and the case $0<s<1$ follows by interpolation.
Then \eqref{eqn:fde3} is equivalent to the complementarity condition % \cite[Chapter 4]{ItoKunisch:2008}
\begin{equation}\label{equiv-cond-q-z}
    q= P_{\Uad} \left(-\gamma^{-1} z\right).
\end{equation}

Now we give higher regularity of the triple $(u,z,q)$.
\begin{lemma}\label{Lemma-Reg-quz}
For any $s\in(0,1/2)$, let $u_d\in H^{s}(0,T;L^2(\Omega))$ and $f\in H^{s}(0,T;L^2(\Omega))$.
Then the solution $(u,z,q)$ of problem \eqref{eqn:fde1}--\eqref{eqn:fde3} satisfies the following estimate
\begin{equation*}
  \|q\|_{H^{\min(1,\alpha+s)} (0,T;L^2(\Omega))} +\|u\|_{H^{\alpha+s}(0,T;L^2(\Omega))}+
  \|z\|_{H^{\alpha+s}(0,T;L^2(\Omega))}\leq c.
\end{equation*}\vspace{-.10mm}
\end{lemma}
\begin{proof}
Let $r=\min(1,\alpha+s)$. By \eqref{equiv-cond-q-z} and \eqref{eqn:bdd-Puad}, we have
\begin{equation*}
  \|q\|_{H^{r} (0,T;L^2(\Omega))}\le c\|z\|_{H^{r} (0,T;L^2(\Omega))}  \le c\|z \|_{H^{\alpha+s} (0,T;L^2(\Omega))}.
\end{equation*}
Applying Theorem \ref{thm:reg} to \eqref{eqn:fde2} yields
\begin{equation*}
  \begin{aligned}
 \|z\|_{H^{\alpha+s} (0,T;L^2(\Omega))} & \le c(\|u\|_{H^{s}(0,T;L^2(\Omega))}+\|u_d\|_{H^{s}(0,T;L^2(\Omega))})
  \le c \|u\|_{H^{s} (0,T;L^2(\Omega))}+c.
  \end{aligned}
\end{equation*}
Similarly, applying Theorem \ref{thm:reg} to \eqref{eqn:fde1} gives
\begin{equation}\label{Reg-uh}
  \begin{aligned}
\|u\|_{H^{\alpha+s} (0,T;L^2(\Omega))}  & \le c(\|f\|_{H^{s} (0,T;L^2(\Omega))}+\|q\|_{H^s(0,T;L^2(\Omega))})
   \le c+c\|q\|_{H^s(0,T;L^2(\Omega))} .
  \end{aligned}
\end{equation}
The last three estimates together imply
\begin{equation*}
  \begin{aligned}
  \|q\|_{H^{r} (0,T;L^2(\Omega))}
  & \le c+c\|q\|_{H^s(0,T;L^2(\Omega))}
  \le c+c_{\epsilon'} \|q\|_{L^2(0,T;L^2(\Omega))}  +\epsilon'\|q\|_{H^r (0,T;L^2(\Omega))} ,
   \end{aligned}
\end{equation*}
where the last step is due to the interpolation inequality \cite[Lemma 24.1]{Tartar:2006}
\begin{equation*}
  \|q\|_{H^s(0,T;L^2(\Omega))}   \le c_{\epsilon'} \|q\|_{L^2(0,T;L^2(\Omega))}  +\epsilon'\|q\|_{H^{r} (0,T;L^2(\Omega))}.
\end{equation*}
By choosing a small $\epsilon'>0$ and the pointwise boundedness of $q$,
cf. \eqref{equiv-cond-q-z}, we obtain
\begin{equation}\label{Reg-qh}
  \|q\|_{H^r (0,T;L^2(\Omega))}
  \le  c+c_{\epsilon'} \|q\|_{L^2(0,T;L^2(\Omega))} \le c.
\end{equation}
This shows the bound on $q$.  \eqref{Reg-qh} and \eqref{Reg-uh} give
the bound on $u$, and that of $z$ follows similarly.
\end{proof}

Next, we give an improved stability estimate on $q$.

\begin{lemma}\label{Lemma-Reg-q}
Let $p>{1}/{\alpha}$ be sufficiently large so that $\alpha\in(0,{1}/{p'})$, $u_d\in
W^{\alpha,p}(0,T;L^2(\Omega))$ and $f\in L^p(0,T;L^2(\Omega))$.
Then the optimal control $q$ satisfies:
\begin{equation*}
  \|q\|_{W^{1/p+\alpha-\epsilon,p}(0,T;L^2(\Omega))} \leq c,
\end{equation*}
where the constant $c$ depends on $\|u_d\|_{W^{\alpha,p}(0,T;L^2(\Omega))}$ and $\|f\|_{L^p(0,T;L^2(\Omega))}$.
\end{lemma}
\begin{proof}
The condition $\alpha\in(0,{1}/{p'})$ implies $r:=1/p+\alpha-\epsilon<1$. Thus
\eqref{eqn:bdd-Puad} and Theorem \ref{thm:reg} (with $s=r-\alpha$) imply
\begin{equation*}
\begin{aligned}
\|q\|_{W^{r,p}(0,T;L^2(\Omega))}
\le & c \|z\|_{W^{r,p}(0,T;L^2(\Omega))}
\le  c\|u-u_d\|_{W^{r-\alpha,p}(0,T;L^2(\Omega))} ,
\end{aligned}
\end{equation*}
Since $p>{1}/{\alpha}$, $r-\alpha={1}/{p}-\epsilon<\alpha$ and thus Theorem \ref{thm:reg} (with $s=0$)
and \eqref{equiv-cond-q-z} give
\begin{equation*}
\|u-u_d\|_{W^{r-\alpha,p}(0,T;L^2(\Omega))}
\le  c\|u-u_d\|_{W^{\alpha,p}(0,T;L^2(\Omega))}
\le  c\|f+q\|_{L^{p}(0,T;L^2(\Omega))}+c\le c,
\end{equation*}
%where the second last inequality is due to Theorem \ref{thm:reg} with $s=0$, and the last inequality due to the pointwise boundedness of $q$, cf.
%\eqref{equiv-cond-q-z}.
The last two estimates together imply the desired result.
\end{proof}

\subsection{Spatially semidiscrete scheme}
Now we give a spatially semidiscrete scheme for problem \eqref{eqn:ob}--\eqref{eqn:fde}: find $q_h\in \Uad$ such that
\begin{equation}\label{eqn:ob-semi}
   \min_{q_h\in \Uad} J(u_h, q_h) = \tfrac12\| u_h - u_d  \|_{L^2(0,T;L^2\II)}^2 + \tfrac\gamma2\| q_h \|_{L^2(0,T;L^2\II)}^2,
\end{equation}
subject to the semidiscrete problem
\begin{equation}\label{eqn:fde-semi}
   \Dal u_h-\Delta_h u_h= P_h(f+q_h), \quad 0<t\leq T,\quad \mbox{with}\quad u_h(0)=0.
\end{equation}
Similar to Theorem \ref{thm:nec-opt}, problem \eqref{eqn:ob-semi}-\eqref{eqn:fde-semi} admits a unique solution
$q_h\in \Uad$. The first-order optimality system reads:
\begin{align}
   \Dal u_h-\Delta_h u_h&= P_h(f+q_h),\quad 0<t\leq T,\quad \mbox{with} \quad u_h(0)=0,\label{eqn:fde1-semi}\\
   _t\partial_T^\alpha z_h-\Delta_h z_h&= P_h(u_h-u_d),\quad 0\leq t<T,\quad \mbox{with}\quad z_h(T)=0, \label{eqn:fde2-semi}\\
     (\gamma q_h + z_h,&v-q_h)_{L^2(0,T;L^2(\Omega))}\ge 0, \quad \forall v\in \Uad.\label{eqn:fde3-semi}
\end{align}
The variational inequality \eqref{eqn:fde3-semi} is equivalent to
\begin{equation}\label{semi-contr-eqv}
    q_h=P_{\Uad} \left(-\gamma^{-1} z_h\right).
\end{equation}

For the approximation \eqref{eqn:fde1-semi}--\eqref{eqn:fde3-semi}, see \cite[Theorem 4.6]{ZhouGong:2016} for an error estimate.
\begin{theorem}\label{thm:err-space}
For $f, u_d\in L^2(0,T;L^2(\Omega))$, let $(u,z,q)$ and $(u_h,z_h,q_h)$ be the solutions of
problems \eqref{eqn:fde1}--\eqref{eqn:fde3} and \eqref{eqn:fde1-semi}--\eqref{eqn:fde3-semi}, respectively.
Then there hold
\begin{equation*}
  \begin{aligned}
  \| u- u_h  \|_{L^2(0,T;L^2(\Omega))} + \| z-z_h  \|_{L^2(0,T;L^2(\Omega))} + \| q-q_h\|_{L^2(0,T;L^2(\Omega))} & \le c  h^2,\\
  \|\nabla( u- u_h) \|_{L^2(0,T;L^2(\Omega))} + \|\nabla( z-z_h  )\|_{L^2(0,T;L^2(\Omega))}  & \leq ch.
  \end{aligned}
\end{equation*}
\end{theorem}

Next, we present the regularity of the semidiscrete solution $(u_h,z_h,q_h)$. The proof is similar
to the continuous case in Lemmas \ref{Lemma-Reg-quz} and \ref{Lemma-Reg-q} and hence omitted.
\begin{lemma}\label{lem:reg-dis}
Let $s\in(0,1/2)$, $u_d\in H^{s}(0,T;L^2(\Omega))$ and $f\in H^{s}(0,T;L^2(\Omega))$.
Then the solution $(u_h,z_h,q_h)$ of problem \eqref{eqn:fde1-semi}--\eqref{eqn:fde3-semi} satisfies the following estimate:
\begin{equation*}
  \|q_h \|_{H^{\min(1,\alpha+s)} (0,T;L^2(\Omega))} +\|u_h\|_{H^{\alpha+s}(0,T;L^2(\Omega))}+
  \|z_h\|_{H^{\alpha+s}(0,T;L^2(\Omega))}\leq c.
\end{equation*}
Further, for $u_d\in W^{\alpha,p}(0,T;L^2(\Omega))$, $f\in L^p(0,T;L^2(\Omega))$, with $p>{1}/{\alpha}$ and $\alpha\in(0,{1}/{p'})$, there holds
\begin{equation*}
  \|q_h\|_{W^{1/p+\alpha-\epsilon,p}(0,T;L^2(\Omega))} \leq c.
\end{equation*}
\end{lemma}

Last, we derive a pointwise-in-time error estimate.
\begin{theorem}\label{thm:err-space-inf}
For $f,~ u_d\in H^1(0,T;L^2(\Omega))$, let $(u,z,q)$ and $(u_h,z_h,q_h)$ be the solutions
of problems \eqref{eqn:fde1}--\eqref{eqn:fde3} and \eqref{eqn:fde1-semi}--\eqref{eqn:fde3-semi}, respectively.
Then there holds
\begin{equation*}
  \| u- u_h  \|_{L^\infty(0,T;L^2(\Omega))} + \| z-z_h  \|_{L^\infty(0,T;L^2(\Omega))} + \| q-q_h\|_{L^\infty(0,T;L^2(\Omega))} \le c \ell_h^2 h^2
\end{equation*}
with $\ell_h=\log(2+1/h)$, where the constant $c$ depends on $\|f\|_{H^1(0,T;L^2(\Omega))}$ and $\|u_d\|_{H^1(0,T;L^2(\Omega))}$.
\end{theorem}
\begin{proof}
We employ the splitting
$u - u_h = (u - u_h(q)) + (u_h(q) - u_h):=\varrho+\vartheta,$
where $u_h(q)\in X_h$ solves
\begin{equation*}
 \partial_t^\alpha u_h(q) -\Delta_h u_h(q) = P_h(f + q),\quad 0<t\leq T,\quad \text{with}\quad u_h(q)(0)=0.
\end{equation*}
Then $u_h(q)$ is the semidiscrete solution of \eqref{eqn:fde-forward} with $g=f+q$, and $\varrho$
is the FEM error for the direct problem. By \cite[Theorem 3.7]{JinLazarovPasciakZhou:2015} and Lemma \ref{Lemma-Reg-q}, we have
\begin{equation}\label{est-varrho}
 \| \varrho  \|_{L^\infty(0,T;L^2\II)} \le c \ell_h^2  h^2\|f+q\|_{L^\infty(0,T;L^2(\Omega))}\leq c\ell_h^2h^2.
\end{equation}
Since $\vartheta$ satisfies $\partial_t^\alpha \vartheta -\Delta_h \vartheta= P_h (q-q_h)$, for $0<t\leq T$ with $\vartheta(0)=0$,
\eqref{eqn:max-lp-h}, $L^2(\Omega)$-stability of $P_h$, the conditions \eqref{equiv-cond-q-z}
and \eqref{semi-contr-eqv}, and the pointwise contractivity of $P_{\Uad}$ imply
\begin{equation}\label{est-vartheta-1}
\begin{aligned}
 &\|{_0\partial_t^\alpha}\vartheta  \|_{L^p(0,T;L^2(\Omega))}
\le \| P_h (q -  q_h)  \|_{L^p(0,T;L^2(\Omega))} \\
\le& c \|  q - q_h  \|_{L^p(0,T;L^2(\Omega))} \le c \| z-z_h \|_{L^p(0,T;L^2(\Omega))}  .
\end{aligned}
\end{equation}
Next, it follows from \eqref{eqn:fde2}, \eqref{eqn:fde2-semi} and the identity $P_h\Delta=\Delta_hR_h$
(with $R_h:H^1(\Omega)\rightarrow X_h$ being Ritz projection) that $w_h:=P_hz-z_h$ satisfies $w_h(T)=0$
\begin{align*}
_t\partial_T^\alpha w_h-\Delta_h w_h= P_hu-u_h
-\Delta_h(P_hz-R_hz),\quad 0\leq t<T,
\end{align*}
and thus
\begin{align}
_t\partial_T^\alpha \Delta_h^{-1}w_h-\Delta_h \Delta_h^{-1}w_h&= \Delta_h^{-1}(P_hu-u_h) - (P_hz-R_hz) .
\end{align}
The maximal $L^p$ regularity \eqref{eqn:max-lp-h} and triangle inequality imply
\begin{align*}
\|w_h\|_{L^p(0,T;L^2(\Omega))}
&\le c\|\Delta_h^{-1}(P_hu-u_h) - (P_hz-R_hz)\|_{L^p(0,T;L^2(\Omega))} \\
&\le c\|P_hu-u_h\|_{L^p(0,T;L^2(\Omega))}
     +c\|P_hz-R_hz\|_{L^p(0,T;L^2(\Omega))}.
\end{align*}
The $L^2(\Omega)$-stability of $P_h$ and triangle inequality yield
\begin{align*}
 \|P_hu-u_h\|_{L^p(0,T;L^2(\Omega))} \le c\|u-u_h\|_{L^p(0,T;L^2(\Omega))} \leq c(\|\vartheta\|_{L^p(0,T;L^2(\Omega))}+\|\varrho\|_{L^p(0,T;L^2(\Omega))}),
\end{align*}
and Theorem \ref{thm:reg} (with $s=0$) and lemma \ref{lem:reg-dis} give
\begin{align*}
\|P_hz-R_hz\|_{L^p(0,T;L^2(\Omega))}
\le c\|z\|_{L^p(0,T;H^2(\Omega))} h^2 \leq c\|u-u_d\|_{L^p(0,T;L^2(\Omega))} h^2\leq ch^2.
\end{align*}
% $L^2(\Omega)$-stability of $P_h$ and Theorem \ref{thm:reg} (with $s=0$) give
%\begin{align*}
%\|w_h\|_{L^p(0,T;L^2(\Omega))}
%&\le c\|\Delta_h^{-1}(P_hu-u_h) - (P_hz-R_hz)\|_{L^p(0,T;L^2(\Omega))} \\
%&\le c\|P_hu-u_h\|_{L^p(0,T;L^2(\Omega))}
%     +c\|P_hz-R_hz\|_{L^p(0,T;L^2(\Omega))} \\
%&\le c\|u-u_h\|_{L^p(0,T;L^2(\Omega))}+c\|z\|_{L^p(0,T;H^2(\Omega))} h^2 \\
%&\le c\|u-u_h\|_{L^p(0,T;L^2(\Omega))}+c\|u-u_d\|_{L^p(0,T;L^2(\Omega))} h^2\\
%&\leq c(\|\vartheta\|_{L^p(0,T;L^2(\Omega))}+\|\varrho\|_{L^p(0,T;L^2(\Omega))})+ch^2,
%\end{align*}
%where the last line follows from the triangle inequality and Lemma \ref{lem:reg-dis}.
The last three estimates and \eqref{est-varrho} yield
\begin{equation*}
  \|w_h\|_{L^p(0,T;L^2(\Omega))} \le c\|\vartheta\|_{L^p(0,T;L^2(\Omega))}+c\ell_h^2h^2.
\end{equation*}
Thus repeating the preceding argument yields
\begin{equation*}
  \begin{aligned}
   \|z-z_h\|_{L^p(0,T;L^2(\Omega))}&\leq \|z-P_hz\|_{L^p(0,T;L^2(\Omega))}+\|w_h\|_{L^p(0,T;L^2(\Omega))}\\
    &\leq c\|\vartheta\|_{L^p(0,T;L^2(\Omega))}+c\ell_h^2h^2.
  \end{aligned}
\end{equation*}
Substituting it into \eqref{est-vartheta-1} and by Sobolev embedding $W^{\alpha,p}
(0,T;L^2(\Omega))\hookrightarrow L^{p_\alpha}(0,T;L^2(\Omega))$, with the critical exponent $p_\alpha=p/(1-p\alpha)$
if $p\alpha<1$, and $p_\alpha=\infty$ if $p\alpha>1$:
\begin{equation}\label{eqn:est-vartheta-2}
 \|\vartheta  \|_{L^{{p_\alpha}}(0,T;L^2(\Omega))}
 \le c\|{_0\partial_t^\alpha}\vartheta  \|_{L^p(0,T;L^2(\Omega))}
 \le c\|\vartheta\|_{L^p(0,T;L^2(\Omega))}+c\ell_h^2h^2 .
\end{equation}
A finite number of repeated applications of this inequality yields
\begin{equation}\label{eqn:est-vartheta-3}
 \| \vartheta  \|_{L^\infty(0,T;L^2(\Omega))} \le c \|\vartheta \|_{L^2(0,T;L^2(\Omega))}+ c\ell_h^2 h^2  \le c \ell_h^2 h^2,
\end{equation}
where we have used the fact that, by maximal $L^p$ regularity \eqref{eqn:max-lp-h} and Theorem \ref{thm:err-space},
\begin{equation*}
  \|\vartheta\|_{L^2(0,T;L^2(\Omega))}=\|u_h(q)-u_h\|_{L^2(0,T;L^2(\Omega))}\leq c\|q-q_h\|_{L^2(0,T;L^2(\Omega))}\leq ch^2.
\end{equation*}
This gives the desired bound on $\|u-u_h\|_{L^\infty(0,T;L^2\II)}$. The bounds on $\|z - z_h\|_{L^\infty(0,T;L^2\II)}$
and  $\|q  -q_h\|_{L^\infty(0,T;L^2\II)}$ follow similarly by the contraction property of $P_{U_\mathrm{ad}}$.
\end{proof}

\subsection{Fully discrete scheme}
Now we turn to the fully discrete approximation of \eqref{eqn:ob}--\eqref{eqn:fde}, with L1 scheme
or BE-CQ time stepping. First, we define a discrete admissible set
$$
U_{ad}^\tau = \{\vec Q_h=(Q_h^{n-1})_{n=1}^N: a\le Q^{n-1}\le b,\,\,~n=1,2,...,N \},
$$
and consider the following fully discrete problem:
\begin{equation*}
    \min_{\vec Q_h \in U_{ad}^\tau} \frac\tau2\sum_{n=1}^N \Big(\| U_h^n - u_d^n  \|_{L^2\II}^2 + \gamma \| Q_h^{n-1}  \|_{L^2\II}^2 \Big),
\end{equation*}
subject to the fully discrete problem
\begin{equation*}
  \bPtau U_h^n-\Delta_h U_h^n = f_h^n+P_h Q_h^{n-1},  \quad n=1,2,...,N, \quad \text{with} ~U_h^0 =0,
\end{equation*}
with $u_d^n=u_d(t_n)$ and $f_h^n=P_hf(t_n)$.
Let $\bar\partial_\tau^\alpha \varphi^n$ be the L1/BE-CQ approximation of $_t\partial_T^\alpha \varphi (t_n)$:
\begin{equation*}
  \bar\partial_\tau^\alpha \varphi^{N-n}
  =\tau^{-\alpha}\sum_{j=0}^{n} \beta_{n-j} \varphi^{N-j} .
\end{equation*}
Then the fully discrete problem is to find $(U_h^n, Z_h^n, Q_h^n)$ such that
\begin{align}
\bPtau U_h^n-\Delta_h U_h^n &= f_h^n+P_h Q_h^{n-1},  && n=1,2,...,N, \quad \text{with } U_h^0 =0 ,
\label{eqn:fde1-fully} \\
\bar\partial_\tau^\alpha Z_h^{n-1}-\Delta_h Z_h^{n-1} &= U_h^n-P_h u_d^n ,&& n=1,2,...,N,\quad \text{with } Z_h^N =0 ,
\label{eqn:fde2-fully}\\
   (\gamma Q_h^{n-1}  + Z_h^{n-1}  , v-Q_h^{n-1} ) &\ge 0, && \forall\, v\in L^2\II ~~\mbox{s.t.}~~a\le v\le b.\label{eqn:fde3-fully}
\end{align}
Similar to \eqref{eqn:fde3-semi}, \eqref{eqn:fde3-fully} can be rewritten as
\begin{equation}\label{fully-contr-eqv}
Q_h^{n-1}=P_{U_{ad}}\big(-\gamma^{-1}Z_h^{n-1}\big),\quad n=1,2,\ldots,N .
\end{equation}

To simplify the notation, we define a discrete $L^2(0,T;L^2(\Omega))$ inner product $[\cdot,\cdot]_\tau$ by
\begin{align*}
 [{\bf v},{\bf w}]_\tau = \tau\sum_{n=1}^N (v_n,w_{n}) \quad \forall\,{\bf v}=(v_n)_{n=1}^N,
{\bf w}=(w_n)_{n=1}^{N} \in L^2(\Omega)^N ,
\end{align*}
and denote by $\vv\cdot\vv_\tau$ the induced norm.
Let $\bPtau{\bf v}=(\bPtau v^n_h)_{n=1}^N \in L^2(\Omega)^N$ and $\bar\partial_\tau^\alpha{\bf w}=(\bar\partial_\tau^\alpha w^{n-1}_h)_{n=1}^{N} \in L^2(\Omega)^N.$
Then the discrete integration by parts formula holds \cite[Section 5.2]{ZhouGong:2016}
\begin{align}\label{eqn:L1-adjoint}
[\bPtau {\bf v}, {\bf w}]_\tau
=[{\bf v},\bar\partial_\tau^\alpha {\bf w}]_\tau
\quad\forall\, {\bf v} ,{\bf w} \in L^2(\Omega)^N  .
\end{align}
Thus, $\bar\partial_\tau^\alpha$ is the adjoint to $\bPtau$ with respect to $[\cdot,\cdot]_\tau$. Let
\begin{align*}
&{\bf U}_h=(U_h^n)_{n=1}^N, && {\bf Z}_h=(Z_h^{n-1})_{n=1}^N, &&
{\bf Q}_h=(Q_h^{n-1})_{n=1}^N ,\\
&{\bf u}_h=(u_h(t_n))_{n=1}^N, && {\bf z}_h=(z_h(t_{n-1}))_{n=1}^N, &&
{\bf q}_h=(q_h(t_{n-1}))_{n=1}^N .
\end{align*}

Next we introduce two auxiliary problems. Let $\mathbf{U}_h(q_h)=(U_h^n(q_h)) )_{n=1}^N\in X_h^N$ solve
\begin{equation}\label{eqn:Uh_q}
  \bPtau U_h^n(q_h)  -\Delta_hU_h^n(q_h)= f_h^n + q_h(t_{n-1}) ,\quad n=1,\ldots,N,\quad \mbox{with }U_h^0(q_h)=0 .
\end{equation}
By Lemma \ref{lem:reg-dis}, the pointwise evaluation $q_h(t_n)$ does make sense, and thus problem
\eqref{eqn:Uh_q} is well defined. For any ${\bf v}_h=(v_h^n)_{n=1}^N$, let ${\bf Z}_h({\bf v}_h)=
(Z_h^{n-1}({\bf v}_h))_{n=1}^N\in X_h^N$ solve
\begin{equation}\label{Eq-Zh-vh}
  \bar\partial_\tau^\alpha Z_h^{n-1}({\bf v}_h) -\Delta_h Z_h^{n-1}({\bf v}_h)
= v_h^n - P_h u_d^n ,\,\, n=1,2,\ldots,N,\,\, \text{with } Z_h^N({\bf v}_h) =0 .
\end{equation}

The rest of this part is devoted to error analysis. First, we bound $\vv{\bf q}_h-{\bf Q}_h\vv_\tau $.
\begin{lemma}\label{lem:nonneg}
For $\mathbf{Q}_h$, $\mathbf{q}_h$, ${\bf Z}_h$ and ${\bf Z}_h({\bf U}_h(q_h))$ defined as above, there holds
\begin{equation*}
\gamma \vv {\bf Q}_h-{\bf q}_h\vv_\tau ^2
  \le  [{\bf q}_h-{\bf Q}_h ,{\bf Z}_h({\bf U}_h(q_h))- {\bf z}_h]_\tau.
\end{equation*}
\end{lemma}
\begin{proof}
It follows from  \eqref{eqn:fde1-fully} and \eqref{eqn:Uh_q}, similarly from \eqref{eqn:fde2-fully} and \eqref{Eq-Zh-vh}, that
\begin{equation*}
   (\bPtau -\Delta_h) ({\bf U}_h(q_h)-{\bf U}_h) =  {\bf q}_h- {\bf Q}_h\quad\mbox{and}\quad
   (\bar\partial_\tau^\alpha -\Delta_h) ({\bf Z}_h({\bf v}_h)-{\bf Z}_h) = {\bf v}_h - {\bf U}_h.
\end{equation*}
Together with \eqref{eqn:L1-adjoint}, these identities imply
\begin{align}
%\begin{split}
[{\bf q}_h-{\bf Q}_h , {\bf Z}_h-{\bf Z}_h({\bf U}_h(q_h))]_\tau
&=\big[(\bPtau - \Delta_h)({\bf U}_h (q_h)-{\bf U}_h)  , {\bf Z}_h-{\bf Z}_h({\bf U}_h(q_h))\big]_\tau\nonumber \\
&= \big[{\bf U}_h (q_h)-{\bf U}_h  , (\bar\partial_\tau^\alpha - \Delta_h)({\bf Z}_h-{\bf Z}_h({\bf U}_h(q_h)))\big]_\tau\nonumber \\
&= - \vv {\bf U}_h (q_h)-{\bf U}_h \vv_\tau^2\leq0.\label{eqn:nonneg}
%\end{split}
\end{align}
Next, since \eqref{semi-contr-eqv} holds pointwise in time, i.e.,
$ q_h(t_{n-1}) = P_{U_\mathrm{ad}}(-\gamma^{-1}z_h(t_{n-1}))$, we have
\begin{equation}\label{contr-qhtn-1}
  (q_h(t_{n-1})+\gamma^{-1}z_h(t_{n-1}),\chi-q_h(t_{n-1}))\geq 0,\quad \forall\, \chi\in L^2(\Omega) \,\,\, \mbox{s.t.}\,\,\, a\leq \chi \leq b.
\end{equation}
Upon setting $v=q_h(t_{n-1})$ in \eqref{eqn:fde3-fully} and $\chi=Q_h^{n-1}$ in \eqref{contr-qhtn-1}, we deduce
\begin{align*}
  \gamma \vv {\bf Q}_h-{\bf q}_h\vv_\tau ^2 &
  =\gamma[{\bf Q}_h -{\bf q}_h, {\bf Q}_h]_\tau  - \gamma[{\bf Q}_h -{\bf q}_h, {\bf q}_h]_\tau\\
  &\le [{\bf q}_h-{\bf Q}_h , {\bf Z}_h]_\tau - [{\bf q}_h-{\bf Q}_h , {\bf z}_h]_\tau \\
  &= [{\bf q}_h-{\bf Q}_h , {\bf Z}_h-{\bf Z}_h({\bf U}_h(q_h))]_\tau + [{\bf q}_h-{\bf Q}_h ,{\bf Z}_h({\bf U}_h(q_h))- {\bf z}_h]_\tau.
\end{align*}
Now invoking \eqref{eqn:nonneg} completes the proof of the lemma.
\end{proof}

The next result gives an error estimate for the approximate state $\mathbf{U}_h(q_h)$.
\begin{lemma}\label{lem:fully-v3}
Let $f,u_d\in H^1(0,T;L^2(\Omega))$. For any $\epsilon\in(0,\min(1/2,\alpha))$, there holds
\begin{equation*}
  \vv {\bf U}_h(q_h) - {\bf u}_h\vv_\tau
 \leq c\tau^{{1}/{2}+\min({1}/{2},\alpha-\epsilon)} .
\end{equation*}\vspace{-10pt}
\end{lemma}
\begin{proof}
By the triangle inequality, we have
\begin{equation*}
  \vv {\bf U}_h(q_h) - {\bf u}_h\vv_\tau \leq \vv {\bf U}_h(q_h)-\widetilde{\bf U}_h(q_h)\vv_\tau + \vv \widetilde {\bf U}_h(q_h)-{\bf u}_h\vv_\tau,
\end{equation*}
where $\mathbf{\widetilde U}_h(q_h)=(\widetilde{U}_h^n(q_h))_{n=1}^N$ is the solution to
\begin{equation}\label{Eq-Uh-qh}
  \bPtau \widetilde{U}_h^n(q_h)-\Delta_h^n\widetilde U_h^n(q_h)  = f_h^n + \tilde q_h^n, \quad n =1,2,\ldots,N\quad \mbox{with } \widetilde U_h^0(q_h)=0 ,
\end{equation}
with $\tilde q_h^n=P_hq_h(t_n)$ (and $\tilde{\bf q}_h=(\tilde q_h^n)_{n=1}^N$). That is, $\widetilde U_h^n(q_h)$
is the fully discrete solution of problem \eqref{eqn:fem} with
$g=f+q_h$. By Lemmas \ref{lem:max-reg} and \ref{lem:L-interp}, we have
\begin{equation}\label{Uhqh-tildeUhqh}
  \vv {\bf U}_h(q_h)-\widetilde {\bf U}_h(q_h)\vv_\tau\leq c\|{\bf q}_h-\tilde {\bf q}_h\vv_\tau \leq c\tau^{\min({1}/{2}+\alpha-\epsilon,1)}\|q_h\|_{H^{{1}/{2}+\alpha-\epsilon}(0,T;L^2(\Omega))}.
\end{equation}
Further, Theorem \ref{thm:err-time2} (with $s=\min(1,{1}/{2}+\alpha-\epsilon)\in(1/2,1)$) implies
\begin{equation*}
\begin{split}
 \vv \widetilde{\bf U}_h(q_h) - {\bf u}_h\vv_\tau
 &\le c\| P_hf+q_h \|_{ H^s(0,T;L^2(\Omega))} \tau^{s}\\
 &\leq c(\|P_hf\|_{H^s(0,T;L^2(\Omega))}+\|q_h \|_{H^s(0,T;L^2(\Omega))}) \tau^s .
\end{split}
\end{equation*}
The last two estimates and Lemma \ref{lem:reg-dis} (with $s=1/2-\epsilon$) yield the desired assertion.
\end{proof}

Now we can give an $\ell^2(L^2(\Omega))$ error estimate for the approximation $(U_h^n,Z_h^n,Q_h^n)$.
\begin{theorem}\label{thm:full}
For $f\in  H^1 (0,T;L^2\II)$ and $u_d\in  H^1(0,T;L^2(\Omega))$,
$(u_h,z_h,q_h)$ and $(U_h^n,Z_h^n,Q_h^n)$ be the solutions of problems \eqref{eqn:fde1-semi}-\eqref{eqn:fde3-semi}
and \eqref{eqn:fde1-fully}-\eqref{eqn:fde3-fully}, respectively. Then there holds for any small $\epsilon>0$
\begin{equation*}
  \vv {\bf u}_h- {\bf U}_h  \vv_\tau + \vv {\bf z}_h- {\bf Z}_h  \vv_\tau + \vv {\bf q}_h- {\bf Q}_h  \vv_\tau
  \le c \tau ^{1/2+\min(1/2,\alpha-\epsilon)} ,
\end{equation*}
where the constant $c$ depends on $\|f\|_{H^1(0,T;L^2(\Omega))}$
and $\|u_d\|_{H^1(0,T;L^2(\Omega))}$.
\end{theorem}
\begin{proof}
By Lemma \ref{lem:nonneg} and the triangle inequality, we deduce
\begin{equation*}%\label{vvQh-qhvv-2}
\begin{split}
 \vv {\bf Q}_h-{\bf q}_h\vv_\tau
  & \le c \vv {\bf Z}_h({\bf U}_h(q_h))- {\bf Z}_h(\mathbf{ u}_h)\vv_\tau+ c \vv {\bf Z}_h(\mathbf{u}_h)- {\bf z}_h\vv_\tau.
 % & \leq c\vv {\bf U}_h(q_h) - {\bf u}_h \vv_\tau + c \vv {\bf Z}_h(\mathbf{u}_h)- {\bf z}_h\vv_\tau\\
 % &  \le c\tau^{r}
 % + c \vv {\bf Z}_h(\mathbf{ u}_h)- {\bf z}_h\vv_\tau,
\end{split}
\end{equation*}
It suffices to bound the two terms on the right hand side.
Lemmas \ref{lem:max-reg} and \ref{lem:fully-v3} imply
\begin{equation*}
  \vv {\bf Z}_h({\bf U}_h(q_h))- {\bf Z}_h(\mathbf{ u}_h)\vv_\tau
  \leq c\vv {\bf U}_h(q_h) - {\bf u}_h \vv_\tau
   \le c\tau^{r}.
\end{equation*}
with $r={1}/{2}+\min({1}/{2},\alpha-\epsilon)$.
Further, since $\mathbf{Z}_h(\mathbf{u}_h)$ is a fully discrete approximation to $z_h(u_h)$, by
Theorem \ref{thm:err-time2} (with $s=r$) and Lemma \ref{lem:reg-dis}, we have
\begin{equation*}
\vv {\bf Z}_h({\bf u}_h)- {\bf z}_h\vv_\tau\le c\| u_h-P_hu_d \|_{H^r (0,T;L^2(\Omega))}  \tau^r \leq c\tau ^r.
\end{equation*}
Thus, we obtain the estimate $\vv {\bf Q}_h-{\bf q}_h\vv_\tau\le c \tau^r.$
Next, by Lemmas \ref{lem:max-reg} and \ref{lem:fully-v3}, we deduce
\begin{equation*}
\begin{split}
     \vv  {\bf U}_h - {\bf u}_h \vv_\tau
     &\le \vv {\bf U}_h -{\bf U}_h ({\bf q}_h)\vv_\tau+\vv {\bf U}_h ({\bf q}_h) - {\bf u}_h \vv_\tau \\
     &\le c\vv {\bf Q}_h - {\bf q}_h \vv_\tau+\vv {\bf U}_h ({\bf q}_h) - {\bf u}_h \vv_\tau
     \le c \tau^r.
\end{split}
\end{equation*}
Similarly, $\vv  {\bf Z}_h - {\bf z}_h \vv_\tau$ can be bounded by
\begin{equation*}
\begin{split}
     \vv  {\bf Z}_h - {\bf z}_h \vv_\tau
     &\le \vv {\bf Z}_h -{\bf Z}_h ({\bf u}_h)\vv_\tau+\vv {\bf Z}_h ({\bf u}_h) - {\bf z}_h \vv_\tau \\
     &\le c\vv {\bf U}_h - {\bf u}_h \vv_\tau
     +\vv {\bf Z}_h ({\bf u}_h) - {\bf z}_h \vv_\tau\le c\tau^r.
\end{split}
\end{equation*}
This completes the proof of Theorem \ref{thm:full}.
\end{proof}

Last, we give a pointwise-in-time error estimate for the approximation $(U_h^n,Q_h^n,Z_h^n)$.
\begin{theorem}\label{thm:full-Linfty}
For $f,u_d\in W^{1,p} (0,T;L^2\II)\cap H^1(0,T;L^2(\Omega))$, $p>1/\alpha$ with $\alpha\in(0,1/p')$,
let $(u_h,z_h,q_h)$ and $(U_h^n,Z_h^n,Q_h^n)$ be the solutions of problems \eqref{eqn:fde1-semi}-\eqref{eqn:fde3-semi}
and \eqref{eqn:fde1-fully}-\eqref{eqn:fde3-fully}, respectively. Then there holds for any small $\epsilon>0$
\begin{equation*}
\max_{1\le n\le N}
\big(\|u^n_h- U^n_h \|_{L^2(\Omega)} + \|z^{n-1}_h- Z^{n-1}_h \|_{L^2(\Omega)}
+ \|q^{n-1}_h- Q^{n-1}_h \|_{L^2(\Omega)} \big)  \le c \tau^{\alpha-\epsilon} ,
\end{equation*}
where the constant $c$ depends on $\|f\|_{W^{1,\min(p,2)}(0,T;L^2(\Omega))}$ and $\|u_d\|_{W^{1,\min(p,2)}(0,T;L^2(\Omega))}$.
\end{theorem}
\begin{proof}
It follows from \eqref{eqn:fde1-fully} and \eqref{eqn:Uh_q} that $U_h^0-U_h^0(q_h) =0$ and
\begin{equation*}
  \bPtau (U_h^n-U_h^n(q_h)) -\Delta_h(U_h^n-U_h^n(q_h)) =  Q_h^{n-1}-q_h(t_{n-1}),\quad n=1,\ldots,N .
\end{equation*}
By Lemma \ref{lem:max-reg} and the inverse inequality (in time), we obtain for any $1/\alpha< p_1<\infty$
\begin{equation*}
\begin{aligned}
\|\bPtau (U_h^n-U_h^n(q_h))_{n=1}^N\|_{\ell^{p_1}(L^2(\Omega))}
&\le c \|(Q_h^{n-1}-q_h(t_{n-1}))_{n=1}^N\|_{\ell^{p_1}(L^2(\Omega))}  \\
&\le   c\tau^{\min(0,1/p_1-1/2)} \|(Q_h^{n-1}-q_h(t_{n-1}))_{n=1}^N\|_{\ell^2(L^2(\Omega))}.
\end{aligned}
\end{equation*}
This and Theorem \ref{thm:full} imply
\begin{equation*}%\label{eqn:est-frac}
\|\bPtau (U_h^n-U_h^n(q_h))_{n=1}^N\|_{\ell^{p_1}(L^2(\Omega))}
\le  c\tau^{\min(1/p_1,1/2) +\min(1/2,\alpha-\epsilon)}.
\end{equation*}
By choosing $p_1>{1}/{\alpha}$ sufficiently close to ${1}/{\alpha}$ and
discrete embedding \cite{JinLiZhou:nonlinear},
\begin{equation*}%\label{eqn:est-infty}
\begin{aligned}
\|(U_h^n-U_h^n(q_h))_{n=1}^N\|_{\ell^\infty(L^2(\Omega))}
&\le c\|\bPtau (U_h^n-U_h^n(q_h))_{n=1}^N\|_{\ell^{p_1}(L^2(\Omega))}\\
&\le c\tau^{\min(1/p_1,1/2)  +\min(1/2,\alpha-\epsilon)}
\le c\tau^{\alpha-\epsilon}  ,
\end{aligned}
\end{equation*}
where the last inequality follows from the inequality $\min(1/p_1,1/2)  +\min(1/2,\alpha-\epsilon)\ge \alpha-\epsilon$, due to the choice
of $p_1$. Further, by the definition of $\widetilde U_h^n(q_h)$ in  \eqref{Eq-Uh-qh}, choosing $p_2>{1}/{\alpha}$
sufficiently large so that $\alpha\in(0,{1}/{p_2'})$ and applying \eqref{l2L2Uhn-00} and Lemma \ref{lem:reg-dis}, we get
\begin{equation*}
\begin{aligned}
\|u_h(t_n)-\widetilde U_h^n(q_h)\|_{L^2(\Omega)}
\le &\,  c\tau^{\alpha-\epsilon} \|f+q_h\|_{W^{1/p_2+\alpha-\epsilon,p_2}(0,T;L^2(\Omega))}   \\
\le &\,   c\tau^{\alpha-\epsilon}(\|q_h\|_{W^{1/p_2+\alpha-\epsilon,p_2}(0,T;L^2(\Omega))} +c) \leq c\tau^{\alpha-\epsilon}.
\end{aligned}
\end{equation*}
Last, by choosing $p_3>1/\alpha$ sufficiently close to $1/\alpha$, Lemmas \ref{lem:max-reg}, \ref{lem:reg-dis},
and \ref{lem:L-interp}, and discrete embedding \cite{JinLiZhou:nonlinear}, we obtain
\begin{align*}
 \|(\widetilde U_h^n(q_h)-U_h^n(q_h))_{n=1}^N\|_{\ell^{\infty}(L^2(\Omega))}&\leq c \|\bPtau(\widetilde U_h^n(q_h)-U_h^n(q_h))_{n=1}^N\|_{\ell^{p_3}(L^2(\Omega))}\\
   &\leq c\|(q_h(t_{n-1})-q_h(t_n))_{n=1}^N\|_{\ell^{p_3}(L^2(\Omega))}\leq c\tau^{\alpha-\epsilon}.
\end{align*}
The last three estimates yield the desired bound on $\|u_h(t_n)-U_h^n\|_{L^2(\Omega)}$.
The bound on $\|z_h(t_{n-1})-Z_h^{n-1}\|_{L^2(\Omega)}$ follows similarly, and that
on $\|q_h(t_{n-1})-Q_h^{n-1}\|_{L^2(\Omega)}$ by the contraction property of $P_{\Uad}$.
\end{proof}

\section{Numerical results and discussions}\label{sec:numer}
Now we present numerical experiments to illustrate the theoretical findings. We perform experiments on the unit
interval $\Omega=(0,1)$. The domain $\Omega$ is divided into $M$ equally spaced subintervals with a mesh size
$h = 1/M$. To discretize the fractional derivatives $_0\partial_t^\alpha u$ and $_t\partial_T^\alpha z$, we fix the
time stepsize $\tau = T/N$. We present numerical results only for the fully discrete scheme by the Galerkin FEM
in space and the L1 scheme in time, since  BE-CQ gives nearly identical results.

We consider the following two examples to illustrate the analysis.
\begin{itemize}
\item[(a)] $f\equiv0$ and $u_d(x,t)=e^t x(1-x)$.
\item[(b)] $f=(1+\cos(t))\chi_{(1/2,1)}(x)$ and $u_d(x,t)=5e^{t}x(1-x)$.
\end{itemize}
Throughout, unless otherwise specified, the penalty parameter $\gamma$ is set to $\gamma=1$, and the lower and upper bounds $a$ and $b$
in the admissible set $\Uad$ to $a=0$ and $b=0.05$. The final time $T$ is fixed at $T=0.1$. The conditions from Theorems
\ref{thm:err-space-inf}, \ref{thm:full} and \ref{thm:full-Linfty} are satisfied for both examples, and thus the error estimates therein hold.

In Tables \ref{tab:a-spatial} and \ref{tab:b-spatial}, we present the spatial error $e_h(u)$ in the $L^\infty
(0,T;L^2(\Omega))$-norm for the semidiscrete solution $u_h$, defined by
\begin{equation*}
e_h(u) = \max_{1\leq n\leq N}\|u_h(t_n)-u(t_n)\|_{L^2(\Omega)},
\end{equation*}
and similarly for the approximations $z_h$ and $q_h$. The numbers in the bracket denote the theoretical rates. Since
the exact solution to problem \eqref{eqn:fde} is unavailable, we
compute reference solutions on a finer mesh, i.e., the continuous solution $u(t_n)$ with a fixed time step $\tau=T/1000$ and
mesh size $h=1/1280$. The empirical rate for the spatial error $e_h$ is of order $O(h^2)$, which is consistent with the theoretical result in
Theorem \ref{thm:err-space-inf}. For case (a), the box constraint is inactive, and thus the errors for the
control $q$ and adjoint $z$ are identical (since $\gamma=1$).

\begin{table}[hbt!]
\caption{Spatial errors for example (a) with $N=10^4$.}\label{tab:a-spatial}
\vspace{-.3cm}{\setlength{\tabcolsep}{7pt}
	\centering
	\begin{tabular}{|c|c|cccccc|c|}
		\hline
		$\alpha$ &$M$ &$10$ &$20$ & $40$ & $80$ & $160$ &$320$ &rate \\
		\hline
		     &   $e_h(u)$      & 4.57e-6 & 1.14e-6 & 2.86e-7 & 7.14e-8 & 1.79e-8 & 4.47e-8  & 2.00 (2.00)\\
		0.4  &   $e_h(q)$      & 3.38e-5 & 8.46e-6 & 2.12e-6 & 5.29e-7 & 1.32e-7 & 3.31e-8  & 2.00 (2.00)\\
		     &   $e_h(z)$      & 3.38e-5 & 8.46e-6 & 2.12e-6 & 5.29e-7 & 1.32e-7 & 3.31e-8  & 2.00 (2.00)\\
    \hline
		     &   $e_h(u)$      & 2.44e-6 & 6.07e-7 & 1.52e-7 & 3.79e-8 & 9.47e-9 & 2.37e-9  & 2.00 (2.00)\\
	 	0.6  &   $e_h(q)$      & 3.62e-5 & 9.04e-6 & 2.26e-6 & 5.65e-7 & 1.41e-7 & 3.53e-8  & 2.00 (2.00)\\
		     &   $e_h(z)$      & 3.62e-5 & 9.04e-6 & 2.26e-6 & 5.65e-7 & 1.41e-7 & 3.53e-8  & 2.00 (2.00)\\
	\hline
		     &   $e_h(u)$      & 8.93e-7 & 2.21e-7 & 5.52e-8 & 1.40e-8 & 3.45e-9 & 8.62e-10  & 2.00 (2.00)\\
		0.8  &   $e_h(q)$      & 3.92e-5 & 9.81e-6 & 2.45e-6 & 6.14e-7 & 1.53e-7 & 3.83e-8   & 2.00 (2.00)\\
		     &   $e_h(z)$      & 3.92e-5 & 9.81e-6 & 2.45e-6 & 6.14e-7 & 1.53e-7 & 3.83e-8   & 2.00 (2.00)\\
	\hline
   \end{tabular}}
\end{table}

\begin{table}[hbt!]
\caption{Temporal errors for example (a) with $M=50$.}\label{tab:a-l2}
\vspace{-.3cm}{\setlength{\tabcolsep}{7pt}
	\centering
	\begin{tabular}{|c|c|cccccc|c|}
		\hline
		$\alpha$ &$N$ &$1000$ &$2000$ & $4000$ & $8000$ & $16000$ &$32000$ &rate \\
		\hline
		     &   $e_{\tau,2}(u)$      & 1.70e-6 & 9.97e-7 & 5.77e-7 & 3.31e-7 & 1.88e-7 & 1.06e-7  & 0.83 (0.90)\\
		0.4  &   $e_{\tau,2}(q)$      & 2.02e-5 & 1.20e-5 & 7.06e-6 & 4.09e-6 & 2.34e-6 & 1.33e-6  & 0.82 (0.90)\\
		     &   $e_{\tau,2}(z)$      & 2.02e-5 & 1.20e-5 & 7.06e-6 & 4.09e-6 & 2.34e-6 & 1.33e-6  & 0.82 (0.90)\\
    \hline
		     &   $e_{\tau,2}(u)$      & 6.58e-7 & 3.47e-7 & 1.82e-7 & 9.47e-8 & 4.90e-8 & 2.53e-8  & 0.96 (1.00)\\
		 0.6 &   $e_{\tau,2}(q)$      & 8.25e-6 & 4.37e-6 & 2.29e-6 & 1.20e-6 & 6.22e-7 & 3.21e-7  & 0.95 (1.00)\\
		     &   $e_{\tau,2}(z)$      & 8.25e-6 & 4.37e-6 & 2.29e-6 & 1.20e-6 & 6.22e-7 & 3.21e-7  & 0.95 (1.00)\\
	\hline
		     &   $e_{\tau,2}(u)$      & 2.68e-7 & 1.38e-7 & 7.07e-8 & 3.62e-8 & 1.84e-8 & 9.38e-9  & 0.97 (1.00)\\
		0.8  &   $e_{\tau,2}(q)$      & 3.80e-6 & 1.95e-6 & 1.00e-6 & 5.12e-7 & 2.61e-7 & 1.33e-7  & 0.98 (1.00)\\
		     &   $e_{\tau,2}(z)$      & 3.80e-6 & 1.95e-6 & 1.00e-6 & 5.12e-7 & 2.61e-7 & 1.33e-7  & 0.98 (1.00)\\
	\hline
   \end{tabular}}
\end{table}

Next, to examine the convergence in time, we compute the $\ell^2(L^2(\Omega))$ and $\ell^\infty
(L^2(\Omega))$ temporal errors $e_{\tau,2}(u)$ and $e_{\tau,\infty}(u)$ for the fully discrete solutions $U_h^n$,
respectively, defined by
\begin{equation*}
 e_{\tau,2}(u) =   \|(U_h^n-u_h(t_n))_{n=1}^N\|_{\ell^2(L^2(\Omega))}
 \quad \mbox{and}\quad
  e_{\tau,\infty}(u) = \max_{1\le n\le N}\|U_h^n-u_h(t_n)\|_{L^2(\Omega)},
\end{equation*}
and similarly for the approximations $Z_h^n$ and $Q_h^n$.
The reference semidiscrete solution $u_h$ is computed with $h=1/50$ and $\tau=1/(64\times10^4)$.
Numerical experiments show that the empirical rate for the temporal discretization error is of
order $O(\tau^{\min(\frac12+\alpha,1)})$ and
$O(\tau^\alpha)$ in the $\ell^2(L^2(\Omega))$ and $\ell^\infty(L^2(\Omega))$-norms, respectively,
cf. Tables \ref{tab:a-l2}--\ref{tab:a-linf} and \ref{tab:b-l2}--\ref{tab:b-linf}, for
cases (a) and (b). These results agree well with the theoretical predictions from Theorems
 \ref{thm:full} and \ref{thm:full-Linfty}, and thus fully support the
error analysis in Section \ref{sec:error}. In Fig. \ref{fig:R0},
we plot the optimal control $q$, the state $u$ and the adjoint $z$. One clearly observes the weak
solution singularity at $t=0$ for the state $u$ and at $t=T$ for the adjoint $z$. The latter is
especially pronounced for case (b). The weak solution singularity is due to the incompatibility of the
source term with the zero initial/terminal data.

\begin{table}[hbt!]
\caption{Pointwise-in-time temporal errors for example (a) with $M=50$.}\label{tab:a-linf}
\vspace{-.3cm}{\setlength{\tabcolsep}{7pt}
	\centering
	\begin{tabular}{|c|c|cccccc|c|}
		\hline
		$\alpha$ &$N$ &$1000$ &$2000$ & $4000$ & $8000$ & $16000$ &$32000$ &rate \\
		\hline
		     &   $e_{\tau,\infty}(u)$      & 3.47e-5 & 2.78e-5 & 2.20e-5 & 1.72e-5 & 1.33e-5 & 1.03e-5  & 0.37 (0.40)\\
		0.4  &   $e_{\tau,\infty}(q)$      & 4.47e-4 & 3.58e-4 & 2.83e-4 & 2.21e-4 & 1.72e-4 & 1.33e-4  & 0.37 (0.40)\\
		     &   $e_{\tau,\infty}(z)$      & 4.47e-4 & 3.58e-4 & 2.83e-4 & 2.21e-4 & 1.72e-4 & 1.33e-4  & 0.37 (0.40)\\
    \hline
		     &   $e_{\tau,\infty}(u)$      & 5.72e-6 & 3.79e-6 & 2.51e-6 & 1.66e-6 & 1.09e-6 & 7.22e-7  & 0.60 (0.60)\\
		0.6  &   $e_{\tau,\infty}(q)$      & 7.64e-5 & 5.06e-5 & 3.35e-5 & 2.21e-5 & 1.46e-5 & 9.64e-6  & 0.60 (0.60)\\
		     &   $e_{\tau,\infty}(z)$      & 7.64e-5 & 5.06e-5 & 3.35e-5 & 2.21e-5 & 1.46e-5 & 9.64e-6  & 0.60 (0.60)\\
	\hline
		     &   $e_{\tau,\infty}(u)$      & 6.93e-7 & 3.97e-7 & 2.28e-7 & 1.31e-7 & 7.50e-8 & 4.31e-8  & 0.80 (0.80)\\
		0.8  &   $e_{\tau,\infty}(q)$      & 9.85e-6 & 5.65e-6 & 3.24e-6 & 1.86e-6 & 1.07e-6 & 6.13e-7  & 0.80 (0.80)\\
		     &   $e_{\tau,\infty}(z)$      & 9.85e-6 & 5.65e-6 & 3.24e-6 & 1.86e-6 & 1.07e-6 & 6.13e-7  & 0.80 (0.80)\\
	\hline
   \end{tabular}}
\end{table}

\begin{figure}[hbt!]
\centering
\setlength{\tabcolsep}{0pt}
\begin{tabular}{ccc}
\includegraphics[trim = .1cm .1cm .1cm .1cm, clip=true,width=0.33\textwidth]{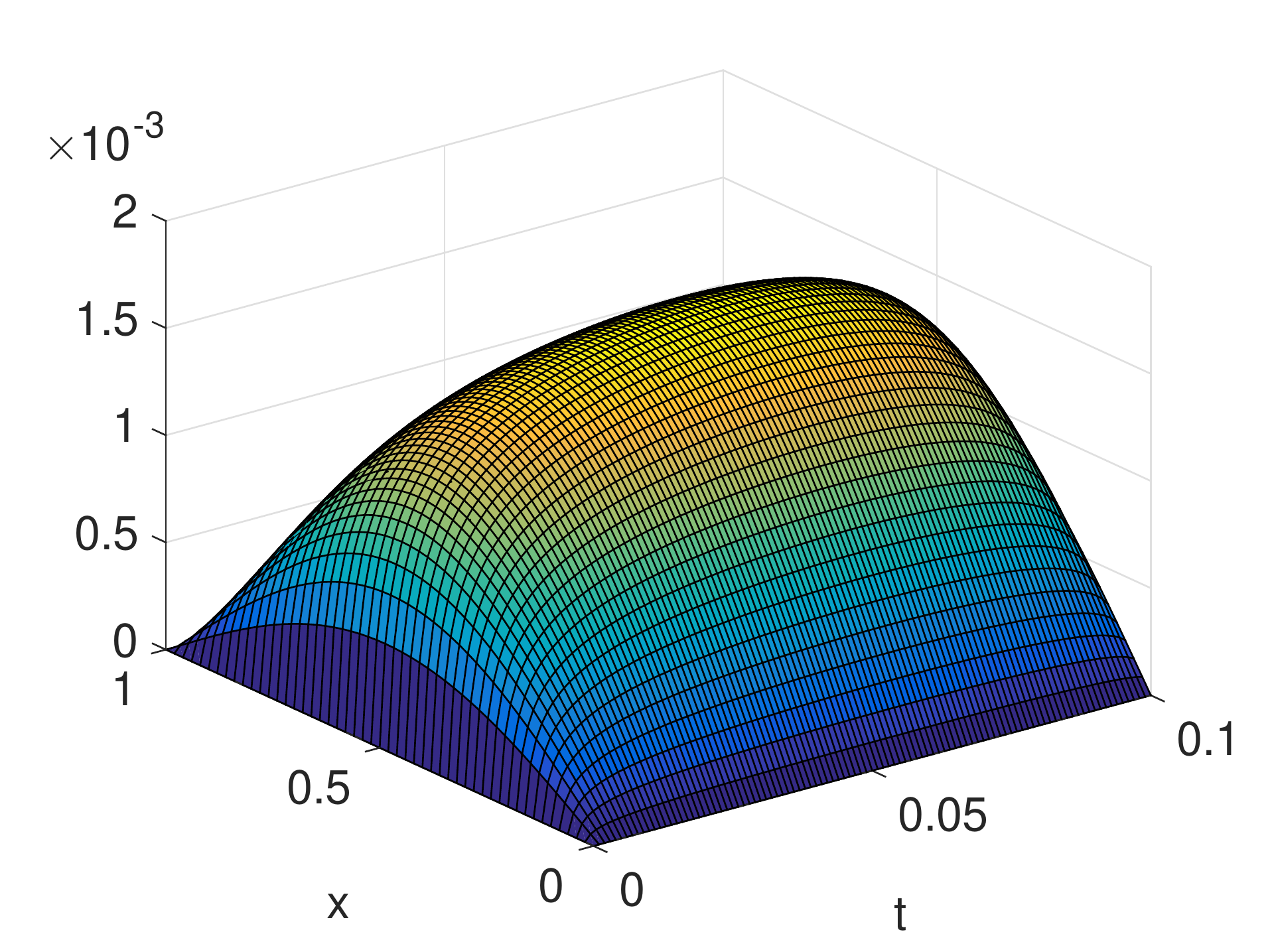}&
\includegraphics[trim = .1cm .1cm .1cm .1cm, clip=true,width=0.33\textwidth]{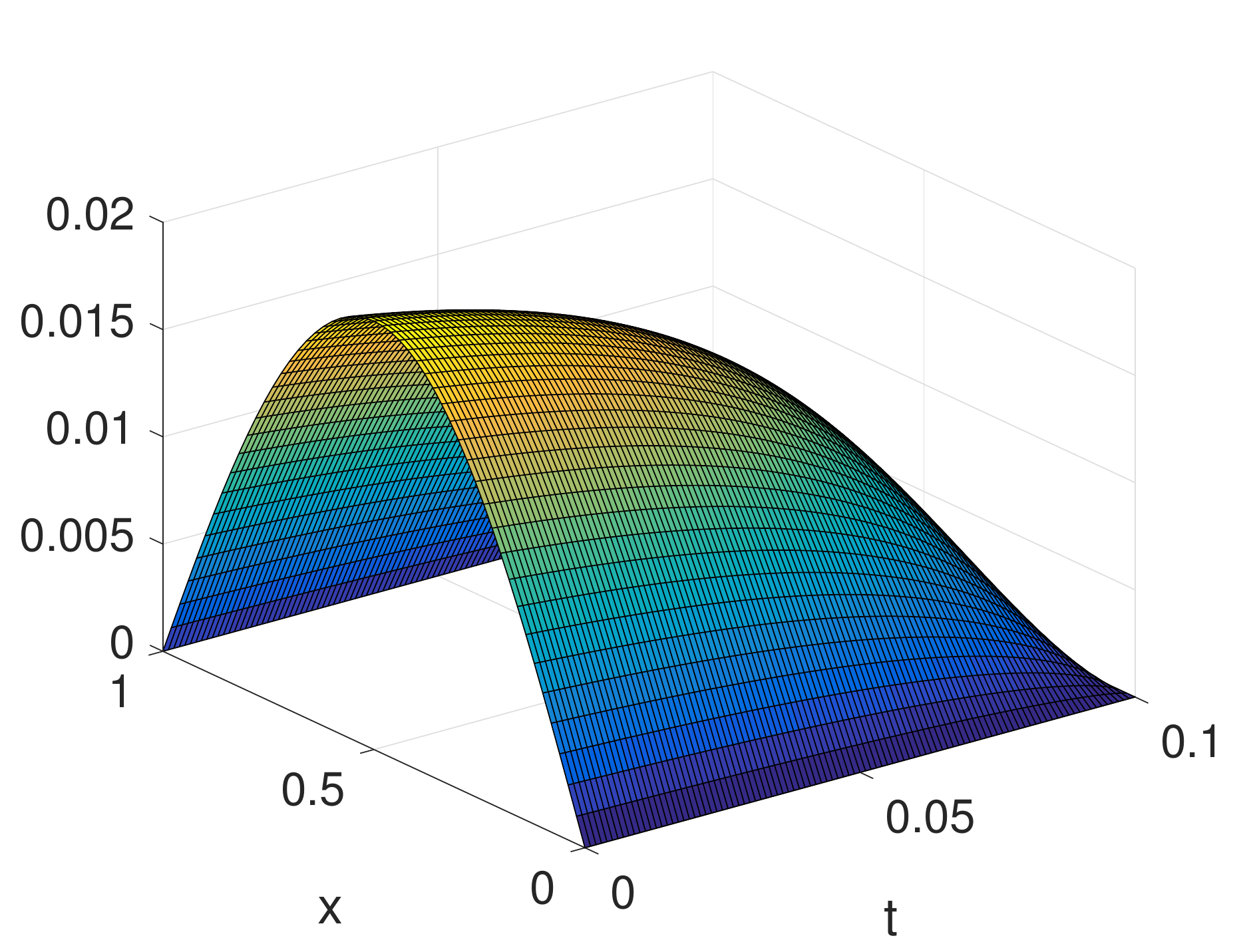}&
\includegraphics[trim = .1cm .1cm .1cm .1cm, clip=true,width=0.33\textwidth]{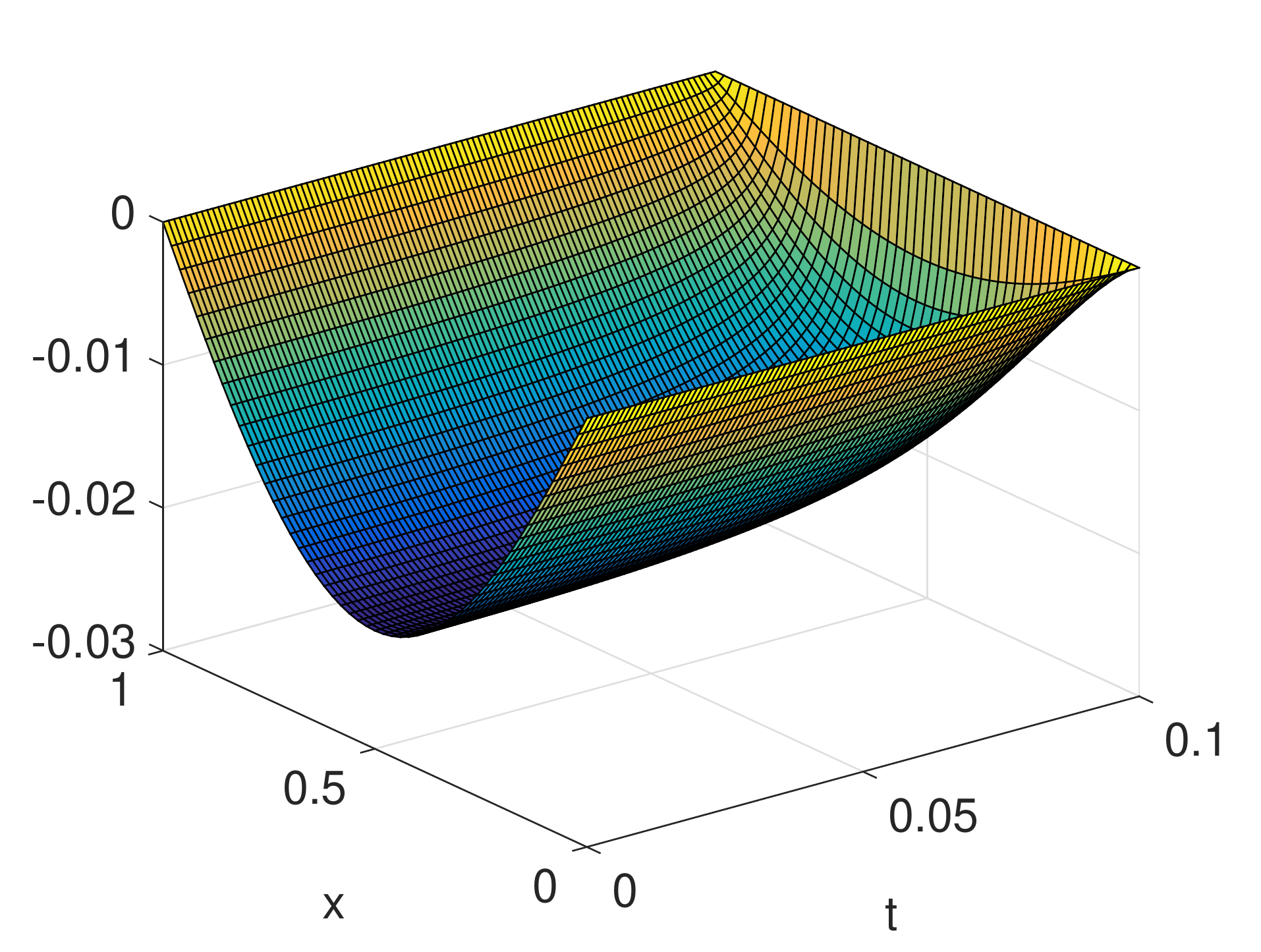}\\
\includegraphics[trim = .1cm .1cm .1cm .1cm, clip=true,width=0.33\textwidth]{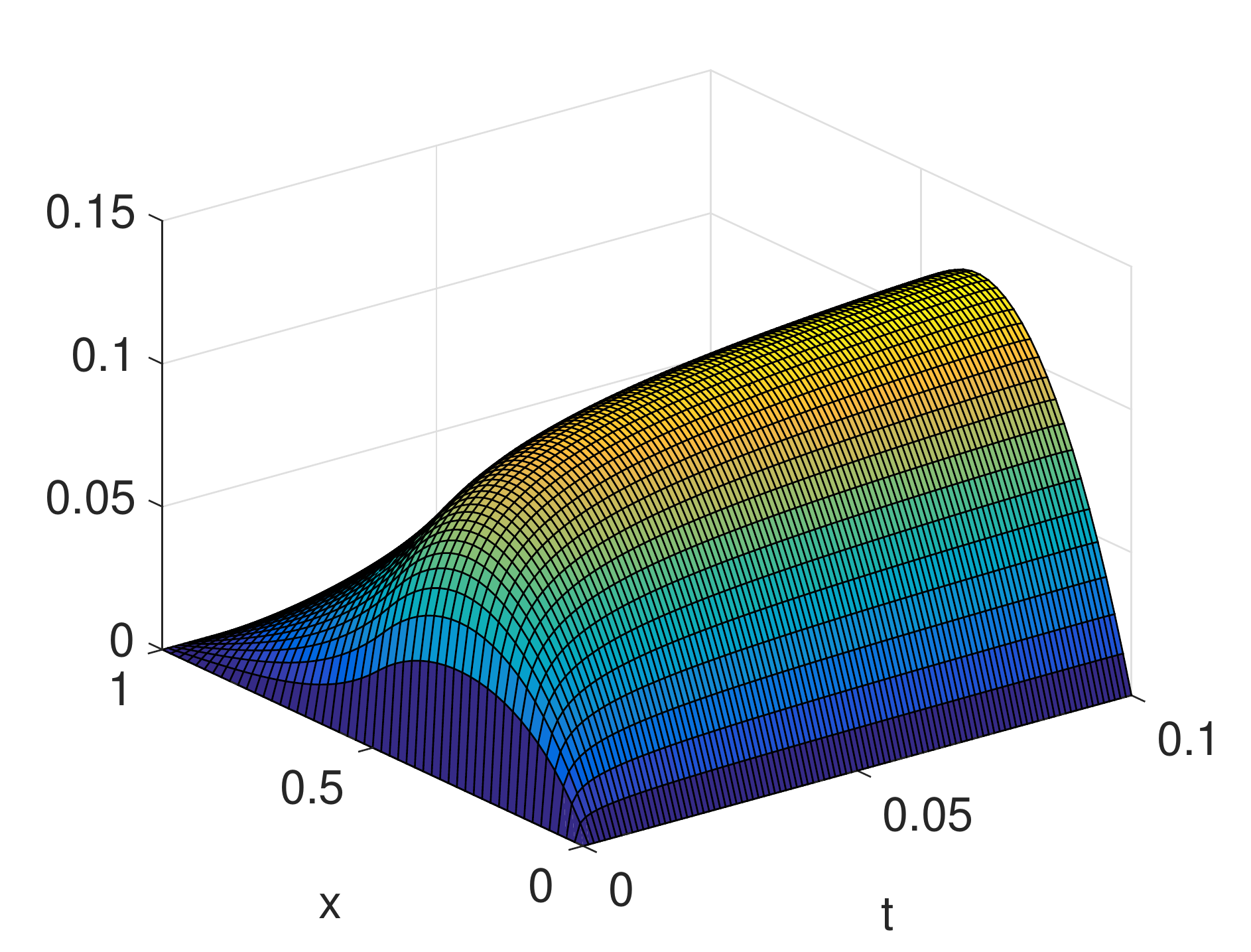}&
\includegraphics[trim = .1cm .1cm .1cm .1cm, clip=true,width=0.33\textwidth]{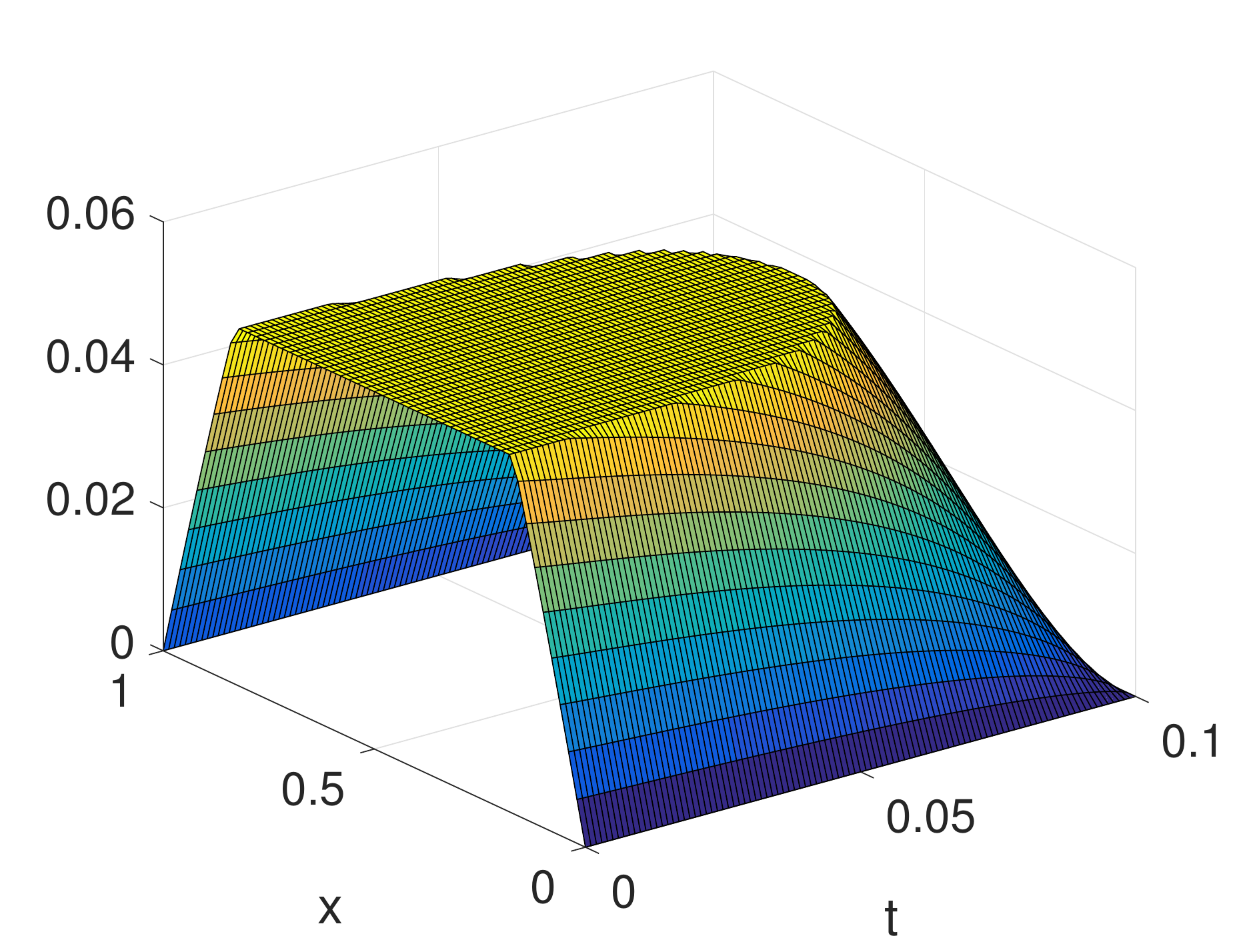}&
\includegraphics[trim = .1cm .1cm .1cm .1cm, clip=true,width=0.33\textwidth]{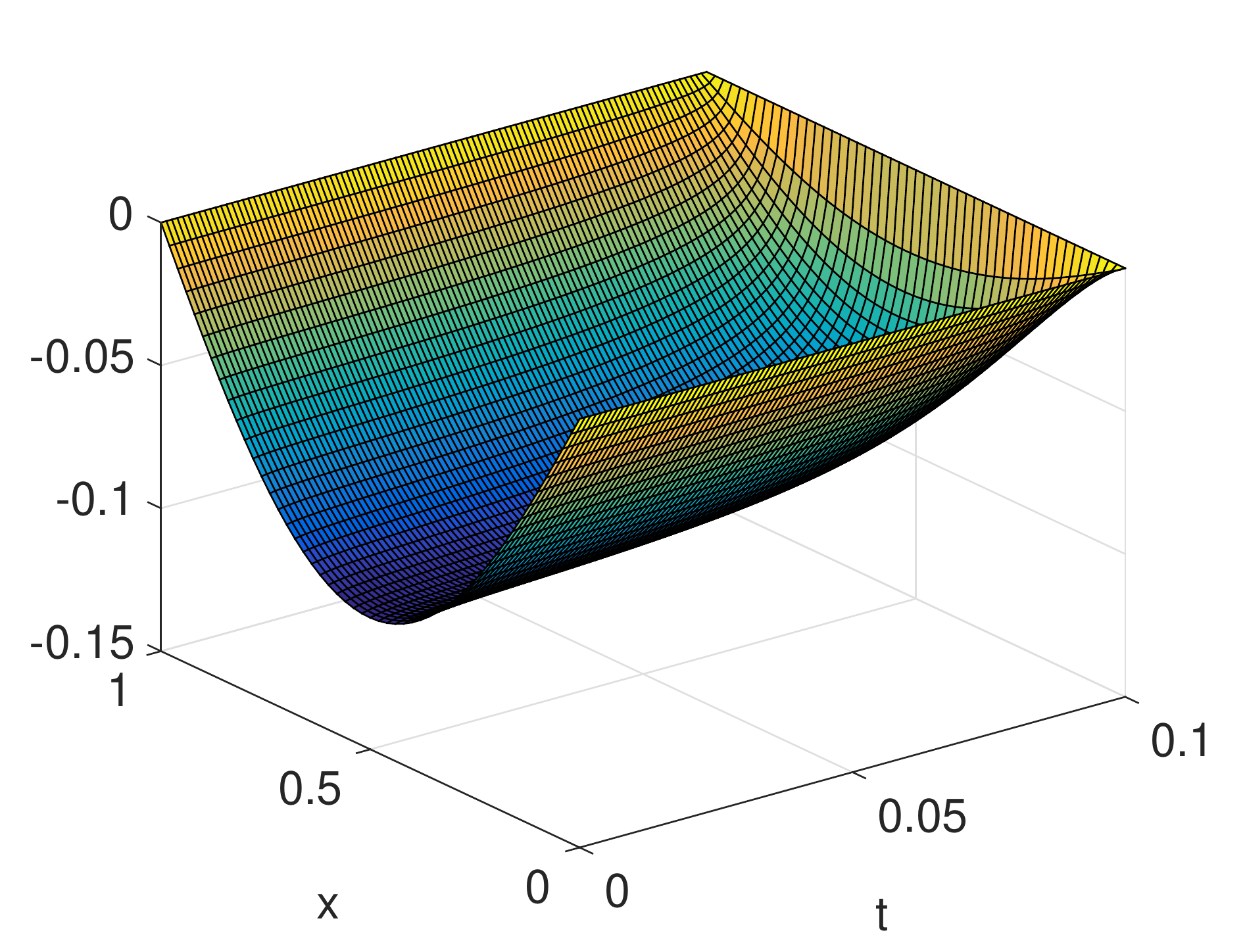}\\
$U_h^n$ & $Q_h^n$ & $Z_h^n$
\end{tabular}
\caption{Plot of $U_h^n$, $Q_h^n$ and $Z_h^n$ for example (a) (top) and (b) (bottom).\label{fig:R0}}
\end{figure}

\begin{table}[hbt!]
\caption{Spatial errors for example (b) with $N=10^4$.}\label{tab:b-spatial}
\vspace{-.3cm}{\setlength{\tabcolsep}{7pt}
	\centering
	\begin{tabular}{|c|c|cccccc|c|}
		\hline
		$\alpha$ &$N$ &$10$ &$20$ & $40$ & $80$ & $160$ &$320$ &rate \\
		\hline
		     &   $e_h(u)$      & 1.86e-4 & 4.72e-5 & 1.16e-5 & 2.82e-6 & 7.39e-7 & 1.84e-7  & 2.00 (2.00)\\
		0.4  &   $e_h(q)$      & 1.59e-4 & 3.97e-4 & 9.92e-6 & 2.48e-6 & 6.19e-7 & 1.55e-7  & 2.00 (2.00)\\
		     &   $e_h(z)$      & 1.78e-4 & 4.44e-5 & 1.11e-5 & 2.78e-6 & 6.94e-7 & 1.74e-7  & 2.00 (2.00)\\
    \hline
		     &   $e_h(u)$      & 1.99e-4 & 4.93e-5 & 1.20e-5 & 3.14e-6 & 7.83e-7 & 1.94e-7  & 2.00 (2.00)\\
		0.6  &   $e_h(q)$      & 1.66e-4 & 4.15e-5 & 1.04e-5 & 2.60e-6 & 6.50e-7 & 1.63e-7  & 2.00 (2.00)\\
		     &   $e_h(z)$      & 1.86e-4 & 4.66e-5 & 1.16e-5 & 2.91e-6 & 7.28e-7 & 1.82e-7  & 2.00 (2.00)\\
    \hline
		     &   $e_h(u)$      & 2.19e-4 & 5.31e-5 & 1.35e-5 & 3.35e-6 & 8.39e-7 & 2.10e-7  & 2.00 (2.00)\\
		0.8  &   $e_h(q)$      & 1.71e-4 & 4.29e-5 & 1.07e-5 & 2.68e-6 & 6.70e-7 & 1.68e-7   & 2.00 (2.00)\\
		     &   $e_h(z)$      & 1.96e-4 & 4.91e-5 & 1.23e-5 & 3.07e-6 & 7.66e-7 & 1.92e-7   & 2.00 (2.00)\\
	\hline
   \end{tabular}}
\end{table}
\begin{table}[hbt!]
\caption{Temporal errors for example (b) with $M=50$.}\label{tab:b-l2}
\vspace{-.3cm}{\setlength{\tabcolsep}{7pt}
	\centering
	\begin{tabular}{|c|c|cccccc|c|}
		\hline
		$\alpha$ &$N$ &$1000$ &$2000$ & $4000$ & $8000$ & $16000$ &$32000$ &rate \\
		\hline
		     &   $e_{\tau,2}(u)$      & 1.05e-4 & 6.34e-5 & 3.79e-5 & 2.24e-5 & 1.31e-5 & 7.59e-6  & 0.79 (0.90)\\
		0.4  &   $e_{\tau,2}(q)$      & 9.00e-5 & 5.43e-5 & 3.21e-5 & 1.87e-5 & 1.07e-5 & 6.10e-6  & 0.81 (0.90)\\
		     &   $e_{\tau,2}(z)$      & 9.36e-5 & 5.57e-5 & 3.26e-5 & 1.89e-5 & 1.08e-5 & 6.15e-6  & 0.82 (0.90)\\
    \hline
		     &   $e_{\tau,2}(u)$      & 4.67e-5 & 2.50e-5 & 1.33e-5 & 6.99e-6 & 3.65e-6 & 1.90e-6  & 0.94 (1.00)\\
		 0.6 &   $e_{\tau,2}(q)$      & 3.66e-5 & 1.95e-5 & 1.03e-5 & 5.40e-6 & 2.81e-6 & 1.45e-6  & 0.95 (1.00)\\
		     &   $e_{\tau,2}(z)$      & 3.83e-5 & 2.03e-5 & 1.07e-5 & 5.56e-6 & 2.89e-6 & 1.49e-6  & 0.95 (1.00)\\
	\hline
		     &   $e_{\tau,2}(u)$      & 2.23e-5 & 1.14e-5 & 5.85e-6 & 2.99e-6 & 1.52e-6 & 7.74e-7  & 0.98 (1.00)\\
		0.8  &   $e_{\tau,2}(q)$      & 1.58e-5 & 8.23e-6 & 4.25e-6 & 2.19e-6 & 1.12e-6 & 5.73e-7  & 0.97 (1.00)\\
		     &   $e_{\tau,2}(z)$      & 1.78e-5 & 9.17e-6 & 4.70e-6 & 2.40e-6 & 1.22e-6 & 6.23e-7  & 0.98 (1.00)\\
	\hline
   \end{tabular}}
\end{table}

\begin{table}[hbt!]
\caption{Pointwise-in-time temporal errors for example (b) with $M=50$.}\label{tab:b-linf}
\vspace{-.3cm}{\setlength{\tabcolsep}{7pt}
	\centering
	\begin{tabular}{|c|c|cccccc|c|}
		\hline
		$\alpha$ &$N$ &$1000$ &$2000$ & $4000$ & $8000$ & $16000$ &$32000$ &rate \\
		\hline
		     &   $e_{\tau,\infty}(u)$      & 2.40e-3 & 1.98e-3 & 1.62e-3 & 1.31e-3 & 1.04e-3 & 8.25e-4  & 0.34 (0.40)\\
		0.4  &   $e_{\tau,\infty}(q)$      & 2.07e-3 & 1.65e-3 & 1.31e-3 & 1.02e-3 & 7.95e-4 & 6.14e-4  & 0.37 (0.40)\\
		     &   $e_{\tau,\infty}(z)$      & 2.07e-3 & 1.65e-3 & 1.31e-3 & 1.02e-3 & 7.95e-4 & 6.14e-4  & 0.37 (0.40)\\
    \hline
		     &   $e_{\tau,\infty}(u)$      & 5.11e-4 & 3.45e-4 & 2.32e-4 & 1.55e-4 & 1.03e-4 & 6.85e-5  & 0.59 (0.60)\\
		0.6  &   $e_{\tau,\infty}(q)$      & 3.54e-4 & 2.35e-4 & 1.55e-4 & 1.03e-4 & 6.77e-5 & 4.47e-5  & 0.60 (0.60)\\
		     &   $e_{\tau,\infty}(z)$      & 3.54e-4 & 2.35e-4 & 1.55e-4 & 1.03e-4 & 6.77e-5 & 4.47e-5  & 0.60 (0.60)\\
	\hline
		     &   $e_{\tau,\infty}(u)$      & 6.96e-5 & 4.03e-5 & 2.33e-5 & 1.34e-5 & 7.68e-6 & 4.41e-6  & 0.80 (0.80)\\
		0.8  &   $e_{\tau,\infty}(q)$      & 4.60e-5 & 2.63e-5 & 1.51e-5 & 8.67e-6 & 4.98e-6 & 2.86e-6  & 0.80 (0.80)\\
		     &   $e_{\tau,\infty}(z)$      & 4.60e-5 & 2.63e-5 & 1.51e-5 & 8.67e-6 & 4.98e-6 & 2.86e-6  & 0.80 (0.80)\\
	\hline
   \end{tabular}}
\end{table}

\section{Conclusions}

In this work, we have developed a complete numerical analysis of a fully discrete scheme for a distributed
optimal control problem governed by a subdiffusion equation, with box constraint on the control
variable, and derived nearly sharp pointwise-in-time error estimates for both space and time discretizations. These estimates agree well with
the empirical rates observed in the numerical experiments. The theoretical and numerical results show  the adverse
influence of the fractional derivatives on the convergence rate when the fractional order $\alpha$ is small.

%Our work represents only a first step towards rigorous numerical analysis for optimal control problems for
%related nonlocal models, and there are several avenues for further research. A first important issue is to
%develop high-order schemes for the optimal control problem. The main
%challenge is to overcome the solution singularities at the initial time for the state variable and at the
%terminal time for the adjoint variable. Second, it is also important to analyze other type of control problems,
%e.g., boundary control and state constraint. Third, the issue of efficient implementation of the fully
%discrete scheme, e.g., semismooth Newton method or preconditioning, is to be addressed.
%Last, it is also interesting to analyze related inverse problems, for which the crucial influence of
%data noise is to be analyzed as well.

\appendix
\section{Proof of Lemma \ref{lem:L-interp}}\label{app:interp}
\begin{proof}
By Sobolev embedding, $ W^{s,p}(0,1;L^2(\Omega))\hookrightarrow C([0,1];L^2(\Omega))$ for $s\in ({1}/{p},1]$, and thus
we can define an interpolation operator $\Pi$ by $\Pi v(\hat t) = v(1)$, for $\hat t\in(0,1)$, for any $v\in W^{s,p}(0,1;L^2(\Omega))$.
The operator $E =I-\Pi$ is bounded from $W^{s,p}(0,1;L^2(\Omega))$ to $L^2(0,1;L^2(\Omega))$:
\begin{equation*}
  \|  Ev\|_{L^2(0,1;L^2(\Omega))} = \| (I-\Pi)v\|_{L^2(0,1;L^2(\Omega))} \le c \| v \|_{W^{s,p}(0,1;L^2(\Omega))}.
\end{equation*}
By the fractional Poincar\'e inequality (cf. \cite{HurriVahakangas:2013}), we have
\begin{equation}\label{eqn:frac-Poincare}
\begin{split}
  \| E v\|_{L^p(0,1;L^2(\Omega))} &=   \inf_{p\in\mathbb{R}}  \| E  (v - p) \|_{L^p(0,1;L^2(\Omega))}
  \le   c \inf_{p\in\mathbb{R}} \| v- p  \|_{W^{s,p}(0,1;L^2(\Omega))} \\
  & \le  c | v|_{W^{s,p}(0,1;L^2(\Omega))} ,
\end{split}
\end{equation}
where the seminorm $|\cdot|_{W^{s,p}(0,T;L^2(\Omega))}$ is defined in \eqref{eqn:SS-seminorm}.
By H\"{o}lder's inequality, we obtain
\begin{equation*}
  \begin{aligned}
    & \| (v(t_n) - \bar v^n)_{n=1}^N  \|_{\ell^p(L^2\II)}  ^p
     = \tau \sum_{n=1}^N \|v(t_n)-\tau^{-1}\int_{t_{n-1}}^{t_n}v(t)\d t\|_{L^2(\Omega)}^p\\
   =&   \tau^{1-p}\sum_{n=1}^N \|\int_{t_{n-1}}^{t_n}(v(t_n)-v(t))\d t\|_{L^2(\Omega)}^p
    \le  \sum_{n=1}^N \int_{t_{n-1}}^{t_n} \|v(t_n) - v(t)\|_{L^2(\Omega)}^p\,\d t.
  \end{aligned}
\end{equation*}
Let $\widehat v_n(\hat t)=v(t_{n-1}+\tau\hat t)$, for $\hat t\in[0,1] $, $n=1,\ldots,N$.
Then $\widehat v_n(\hat t) \in W^{s,p}(0,1;L^2(\Omega))$ and by \eqref{eqn:frac-Poincare}, we have
\begin{align*}
\| (v(t_n) - \bar v^n)_{n=1}^N  \|_{\ell^p(L^2\II)}  ^p
    &\leq \tau \sum_{n=1}^N \int_{0}^{1} \| \Pi \widehat v_n  -  \widehat v_n\|_{L^2(\Omega)}^p( \hat t)\,\d \hat t
    \le c \tau \sum_{n=1}^N  | \widehat v_n|_{W^{s,p}(0,1;L^2(\Omega))}^p \\
    &\le c \tau \sum_{n=1}^N   \int_0^1 \int_0^1 \frac{\|\widehat v_n(\hat t)-\widehat v_n(\hat \xi)\|_{L^2(\Omega)}^p}{|\hat t-\hat \xi|^{1+ps}} \,\d\hat t \d \hat \xi \\
    &= c \tau^{ps} \sum_{n=1}^N   \int_{t_{n-1}}^{t_n}  \int_{t_{n-1}}^{t_n} \frac{\|v(t)-v(\xi)\|_{L^2(\Omega)}^p}{|t-\xi|^{1+ps}} \,\d t\d\xi \\
%    & \leq c\tau^{ps}\sum_{n=1}^N   \int_{t_{n-1}}^{t_n}  \int_{0}^T \frac{\|v(t)-v(\xi)\|_{L^2(\Omega)}^p}{|t-\xi|^{1+ps}} \,dtd\xi \\
    & \le c \tau^{ps}   \int_{0}^T  \int_{0}^T \frac{\|v(t)-v(\xi)\|_{L^2(\Omega)}^p}{|t-\xi|^{1+ps}} \,\d t\d\xi
    =  c \tau^{ps}  |v|_{W^{s,p}(0,T;L^2(\Omega))}^p ,
\end{align*}
which implies the desired assertion.
\end{proof}

\bibliographystyle{abbrv}
\bibliography{frac}

\end{document}